\documentclass{amsart}
\pdfoutput=1
\usepackage{amsmath,enumerate, amssymb, stmaryrd}
\usepackage[colorlinks=true,pagebackref,bookmarksdepth=3]{hyperref}
\usepackage{verbatim}
\usepackage[all]{xy}
\SelectTips{eu}{}
\xyoption{curve}

\CompileMatrices

\hypersetup{
 pdfinfo={
    Title={Lifting free divisors},
   Author={Ragnar-Olaf Buchweitz and Brian Pike}
 }
}

\swapnumbers

\def\lto{{\longrightarrow}}

\def\xto{\xrightarrow}

\newcommand{\fg}{{\mathfrak g}}
\newcommand{\fh}{{\mathfrak h}}
\newcommand{\fm}{{\mathfrak m}}

\newcommand{\calm}{{\mathcal M}}

\newcommand{\calo}{{\mathcal O}}

\newcommand{\calos}{{\calo_S}}

\newcommand{\calox}{{\calo_X}}
\newcommand{\caloy}{{\calo_Y}}

\newcommand{\CC}{{\mathbb C}}

\newcommand{\HH}{{\mathbb H}}

\newcommand{\LL}{{\mathbb L}}

\newcommand{\ZZ}{{\mathbb Z}}

\newcommand{\vp}{\varphi}

\DeclareMathOperator{\adj}{adj}

\DeclareMathOperator{\ann}{ann}

\DeclareMathOperator{\codim}{codim}

\DeclareMathOperator{\coker}{coker}

\DeclareMathOperator{\depth}{depth}
\DeclareMathOperator{\Der}{Der}

\DeclareMathOperator{\End}{ End}

\DeclareMathOperator{\Fitt}{Fitt}
\DeclareMathOperator{\GL}{GL}
\DeclareMathOperator{\gl}{\mathfrak{gl}}

\DeclareMathOperator{\Hom}{Hom}
\DeclareMathOperator{\id}{id}

\DeclareMathOperator{\jac}{jac}
\DeclareMathOperator{\Jac}{Jac}

\DeclareMathOperator{\Orthog}{O}

\DeclareMathOperator{\projdim}{projdim}
\DeclareMathOperator{\Pf}{Pf}

\DeclareMathOperator{\PSL}{PSL}

\DeclareMathOperator{\rank}{rank}

\DeclareMathOperator{\Sing}{Sing}
\DeclareMathOperator{\Sk}{Sk}
\DeclareMathOperator{\SL}{SL}
\DeclareMathOperator{\sltwo}{\mathfrak{sl}_{2}}
\DeclareMathOperator{\sln}{\mathfrak{sl}_{n}}
\DeclareMathOperator{\SO}{SO}
\DeclareMathOperator{\Sp}{Sp}

\DeclareMathOperator{\supp}{ supp}
\DeclareMathOperator{\Sym}{{\mathbb S}ym}
\DeclareMathOperator{\Symm}{Symm}

\DeclareMathOperator{\Tor}{Tor}
\DeclareMathOperator{\tr}{tr}
\DeclareMathOperator{\trace}{trace}

\newcommand{{\sbullet}}{{\scriptstyle\bullet}}

\newcommand{\slashslash}{\ensuremath{\!{/\!/}}}

\theoremstyle{definition}
\newtheorem{defn}{Definition}[section]

\newtheorem{rem}[defn]{Remark}
\newtheorem{remark}[defn]{Remark}

\newtheorem{sit}[defn]{}
\newtheorem{example}[defn]{Example}

\newtheorem{exam}[defn]{Example}

\theoremstyle{plain}
\newtheorem{prop}[defn]{Proposition}
\newtheorem{proposition}[defn]{Proposition}
\newtheorem{theorem}[defn]{Theorem}
\newtheorem{lem}[defn]{Lemma}
\newtheorem{lemma}[defn]{Lemma}
\newtheorem{cor}[defn]{Corollary}
\newtheorem{corollary}[defn]{Corollary}

\def\bexa{\begin{exam}}
\def\eexa{\end{exam}}

\def\bpro{\begin{prop}}
\def\epro{\end{prop}}

\def\bcor{\begin{cor}}
\def\ecor{\end{cor}}

\def\bthm{\begin{theorem}}
\def\ethm{\end{theorem}}

\def\bdfn{\begin{eefn}}
\def\edfn{\end{eefn}}

\def\brem{\begin{rem}}
\def\erem{\end{rem}}

\def\brems{\begin{rems}}
\def\erems{\end{rems}}

\def\bsit{\begin{sit}}
\def\esit{\end{sit}}

\def\blem{\begin{lem}}
\def\elem{\end{lem}}

\def\bdi{\pdfsyncstop\begin{eiagram}}
\def\edi{\end{eiagram}\pdfsyncstart}

\def\ba{\begin{array}}
\def\ea{\end{array}}

\def\bnum{\begin{enumerate}}
\def\enum{\end{enumerate}}

\def\be{\begin{equation}}
\def\ee{\end{equation}}

\def\bproof{\begin{proof}}
\def\eproof{\end{proof}}

\begin{document}
\title{Lifting free divisors}

\author{Ragnar-Olaf Buchweitz}
\address{Dept.\ of Computer and Math\-ematical Sciences,
University  of Tor\-onto at Scarborough, 
1265 Military Trail, 
Toronto, ON M1C 1A4,
Canada}
\email{ragnar@utsc.utoronto.ca}

\author{Brian Pike}
\address{Dept.\ of Computer and Math\-ematical Sciences,
University  of Tor\-onto at Scarborough, 
1265 Military Trail, 
Toronto, ON M1C 1A4,
Canada}
\email{bapike@gmail.com}

\thanks{The first author was partly supported by NSERC grant 3-642-114-80}
\date{\today}

\subjclass[2010]{Primary 32S25; Secondary 17B66, 14L30}

\keywords{free divisors, logarithmic vector fields, discriminants,
coregular group actions, invariants, tangent cohomology,
Kodaira--Spencer map}

\begin{abstract} 
Let $\varphi:X\to S$ be a morphism between smooth complex analytic spaces, and
let $f=0$ define a free divisor on $S$.
We prove that if
the deformation space 
$T^1_{X/S}$ of $\varphi$ is a Cohen--Macaulay $\mathcal{O}_X$--module of codimension 2,
and
all of the logarithmic vector fields for $f=0$ lift via $\varphi$, 
then
$f\circ \varphi=0$ defines a free divisor on $X$;
this is generalized in several directions.

Among applications
we recover a result of Mond--van Straten, generalize a
construction of Buchweitz--Conca, and show that
a map 
$\varphi:\mathbb{C}^{n+1}\to \mathbb{C}^n$ with critical set of codimension $2$
has a $T^1_{X/S}$ with the desired properties.
Finally, if
$X$ is a representation of a reductive complex algebraic group $G$ and
$\varphi$ is the algebraic quotient $X\to S=X\!{/\!/} G$ with
$X\!{/\!/} G$ smooth, we
describe sufficient conditions for $T^1_{X/S}$ to be Cohen--Macaulay of
codimension $2$.
In one such case, 
a free divisor on $\mathbb{C}^{n+1}$ lifts under the
operation of ``castling'' to a free divisor on
$\mathbb{C}^{n(n+1)}$, partially generalizing work of Granger--Mond--Schulze on linear free divisors.
We give several other examples of such representations.
\end{abstract}
\maketitle

{\footnotesize\tableofcontents}

\section{Introduction}
Let $f:S\to (\CC,0)$ be the germ of a holomorphic function defining a
reduced hypersurface germ $D=f^{-1}(0)$ in a smooth complex analytic germ
$S=(\CC^m,0)$.
The $\calos$--module $\Der_S(-\log f)$ of
\emph{logarithmic vector fields}
consists of all germs of holomorphic vector fields on $S$ that are
tangent to the smooth points of $D$.
Then $D$ is called a \emph{free divisor} when
$\Der_S(-\log f)$ is a free $\calos$--module, necessarily of rank $m$,
or equivalently,
when $\Der_S(-\log f)$ requires only $m$ generators, the smallest
number possible. A free divisor is either a smooth hypersurface
or singular in codimension one. The ({\em Saito\/} or {\em discriminant\/}) matrix that describes the inclusion of the
logarithmic vector fields into all vector fields on the ambient space is square
and its determinant is an equation of the free divisor, thus, providing a compact representation
of an otherwise usually dense polynomial or power series.

Free divisors are often `discriminants', broadly interpreted, and then
describe the locus of some type of degenerate behavior.
For instance, free divisors classically arose as discriminants of
versal unfoldings of isolated hypersurface (\cite{Saito}) and isolated complete
intersection singularities (\cite{Looijenga}). As well, the locus in a Frobenius manifold
where the Euler vector field is not invertible is a free divisor (\cite{Hertling}).
More recently, other discriminants have been shown to be
free divisors (e.g.,
\cite{Damon-legacyfree,MvS,Damon-nonlinearsections,BEvB}).

Many hyperplane arrangements are classically known to be free divisors and it is
a long outstanding question whether freeness is a combinatorial property in this case
(\cite{OrlikTerao}).

When $\Der_S(-\log f)$ has a free basis of \emph{linear} vector fields,
then $D$ is a \emph{linear free divisor};
these may be thought of as the
discriminant of a \emph{prehomogeneous vector space},
a representation on $S$ of a linear algebraic group that has a Zariski open
orbit. 
\\

While the above list of examples is meant to highlight that free divisors
are everywhere, and, for example, the assignment from isolated
complete intersection singularities to their discriminants in the base
of a semi--universal deformation is essentially injective by \cite{Wirthmueller},
we still have very few methods to construct such divisors explicitly in a given dimension.

Here we give one approach to such construction.
Let $\vp:X=(\CC^n,0)\to S=(\CC^m,0)$ be a holomorphic map between smooth
complex spaces, and let $D=V(f)$ be a free divisor in $S$.
In this paper we ask:
\begin{center}
\emph{
When is $\vp^{-1}(D)\equiv \{f\vp=0\}\subseteq X$
again a free divisor?}
\end{center}
We give sufficient conditions for $\vp^{-1}(D)$ to be a free divisor, and
describe a number of situations in which these conditions hold.
This gives a flexible method to construct new free divisors, and gives some
insight into the behavior of logarithmic vector fields under this
pullback operation.
\\

The structure of the paper is the following.
Our sufficient conditions are stated in terms of
modules describing the deformations of $\vp$,
and the module of vector fields on $S$ that lift across $\varphi$ to
vector fields on $X$.
Hence,
in \S\ref{sec:background} we introduce 
some deformation theory, the Kodaira-Spencer map, and also free divisors.

\S\ref{sec:mainresult} contains our two main results.
Theorems \ref{altmthm} and \ref{mthm} each
give conditions for $\vp^{-1}(D)$ to be free;
Theorem \ref{mthm} is a consequence of Theorem \ref{altmthm} that
has more restrictive
hypotheses---that are easier to check---and a stronger conclusion.
All but one of our applications use Theorem \ref{mthm}.
Our first example generalizes a result of Mond and van Straten
\cite{MvS}.

Both Theorems require that all vector fields $\eta\in \Der_S(-\log f)$
lift across $\vp$.
In \S\ref{sec:lifteuler}, we relax this condition, at least for the
Euler vector field of a weighted-homogeneous $f$.
The motivation for this consequence of Theorem \ref{mthm}
is the case of $D=\{0\}\subset S=(\CC,0)$,
where the conditions for $\vp^{-1}(D)=V(\vp)$ to define
a free divisor are known and require no lifting of vector fields.

In \S\ref{sec:addcompdim} we describe 
a construction that,
given a free divisor $D$ in $X$ and an
appropriate
ideal $I\subset \calox$, constructs a free divisor on $X\times Y$
that contains $D\times Y$ and has additional nontrivial components.
This application of Theorem \ref{mthm}
generalizes a construction of Buchweitz--Conca \cite{BC},
and a construction for linear free divisors \cite{DP,Pike}.

The rest of the paper
describes situations where the
deformation condition on $\vp$ is satisfied.
Then the requirement for all $\eta\in\Der_S(-\log f)$
to lift is generally satisfied for certain free divisors in $S$.
For instance, in \S\ref{sec:cnp1tocn} we show that maps $\vp:\CC^{n+1}\to\CC^n$
with critical set of codimension $2$ satisfy the deformation
condition.

We begin \S\ref{sec:groupactions} by describing 
our original interest in this problem.
Granger--Mond--Schulze \cite{granger-mond-schulze}
showed that
the set of prehomogeneous vector spaces that define linear free
divisors is
invariant under `castling',
an operation on
prehomogeneous vector spaces.
Thus, the 
corresponding transformation on a
linear free divisor produces another linear free divisor.
In the simplest case,
this transformation is a lift across the map
$\vp:X=M_{n,n+1}\to \CC^{n+1}$,
where $M_{n,n+1}$ is the space of $n\times (n+1)$ matrices and 
each component of $\vp$ is a signed determinant.
In Theorem \ref{thm:maxmin}, we use Theorem \ref{mthm} to
show that pulling back an 
arbitrary free divisor via this $\vp$ produces another free divisor.
This $\vp$ is the algebraic quotient of $\SL(n,\CC)$
acting on $M_{n,n+1}$.

The rest of \S\ref{sec:groupactions} 
generalizes this result by 
studying algebraic quotients
$\vp:X\to S=X\slashslash G$
of a reductive linear algebraic group $G$ acting on $X$.
Since we require that $S$ is smooth, so the action of $G$ is
\emph{coregular},
the components of $\vp$ generate the subring of $G$-invariant
polynomials on $X$.
Proposition \ref{prop:coregulart1cm} gives sufficient conditions
for the deformation condition on $\vp$ to be satisfied,
and Lemma \ref{lemma:grouplift} suggests a method to
prove that certain vector fields are liftable.
As an aside, we point out that the invariants of lowest degree may
be easily computed.

Finally, 
\S\ref{sec:groupexamples} describes many---but not all---examples
of group actions
with quotients satisfying the deformation condition;
for each, we identify those $f$ satisfying the lifting
condition.  This is a very productive method for producing free
divisors.
\\

Acknowledgements:
Discussions with David Mond inspired us to look further at castling,
and Eleonore Faber gave helpful comments on the paper.

\section{Deformation theory and free divisors}
\label{sec:background}

\begin{sit}
If $\vp:X\to S$ is any morphism of complex analytic germs, we write $\vp^{\flat}:\calos\to\calox, 
f\mapsto f\vp$, for the corresponding morphism of local analytic algebras.
By common abuse of notation, we often write
$\vp$ for $\vp^{\flat}$. 
Furthermore, we throughout denote by 
$\fm_{*}$ the maximal ideal in $\calo_{*}$ of germs of functions that vanish at $0$, the distinguished 
point of the germ.
\end{sit}

\subsection*{Deformations}
We begin with some background on deformation theory and tangent cohomology.

\begin{sit}
If $\vp:X\to S$ is still any morphism of analytic germs and $\calm$ an $\calox$--module, we 
denote $T^{i}_{X/S}(\calm)= H^{i}(\Hom_{\calox}(\LL_{X/S},\calm))$, the 
$i^{th}$ tangent cohomology of $X$ over $S$ with values in $\calm$. Here $\LL_{X/S}$ is a cotangent complex
for $\vp$, well defined up to isomorphism in the derived category of coherent $\calox$--modules
(e.g., \cite[Appendix C]{GLS}).

As usual, we abbreviate $T^{i}_{X/S}=T^{i}_{X/S}(\calox)$, and write simply $T^{i}_{X}$ if $\vp$ is 
the constant map to a point.
\end{sit}

\begin{sit}
Note that $T^{0}_{X/S}(\calm)=\Der_{S}(\calox,\calm)$ is the $\calox$--module of $\vp^{-1}(\calos)$--linear
vector fields on $X$ with values in $\calm$, or, shorter, the $\calox$--module of {\em vertical vector fields\/}
along $\vp$ with values in $\calm$. If $\calm=\calox$, we simply speak of the module of vertical vector 
fields along $\vp$.
\end{sit}

\begin{sit}
If $\vp$ is a {\em smooth\/} morphism,
then $T^{i}_{X/S}(\calm)=0$ for all $i\neq 0$, and all $\calm$. As tangent 
cohomology localizes on $X$, the $\calox$--modules
$T^{i}_{X/S}(\calm)$, for $i > 0$, are thus 
supported on the {\em critical locus\/} of $\vp$, the 
closed subgerm $C(\vp)\subseteq  X$, where 
$\vp$ fails to be smooth (e.g., \cite{GLS}). 
\end{sit}

\begin{sit}
\label{sit:ZJ}
If $\vp:X\to S$ is any morphism of analytic germs, it induces the (dual) {\em Zariski--Jacobi 
sequence\/} in tangent cohomology, the long exact sequence of $\calox$--modules
\begin{align*}
\xymatrix{
0\ar[r]&T^{0}_{X/S}\ar[r]&T^{0}_{X}\ar[r]^-{\Jac(\vp)}&T^{0}_{S}(\calox)\ar[r]^{\delta}&T^{1}_{X/S}\ar[r]&
T^{1}_{X}\ar[r]
&\cdots\,,
}
\end{align*}
where $\Jac(\vp)$ is the $\calox$--dual to 
$d\vp:\vp^{*}\Omega^{1}_{S}\cong \calox\otimes_{\calos}\Omega^1_S
  \to \Omega^{1}_{X}$ that in turn sends
$1\otimes_{\calos}ds$ to $d(s\vp)$ for any function germ $s\in\calos$.

If $x_{1},...,x_{n}$ are local coordinates on $X$ and $s_{1},...,s_{m}$ are local coordinates on $S$, then
a vector field 
\begin{align}
Z&=\sum_{i=1}^{n}g_{i}\frac{\partial}{\partial x_{i}}\in\  T^{0}_{X}\ \subseteq\ 
\bigoplus_{i=1}^{n}\calox\frac{\partial}{\partial x_{i}}\,,
\label{al:vfformt0x} 
\intertext{with coefficients $g_{i}\in \calox$, maps to the vector field}
\Jac(\vp)(Z) &= 
\sum_{j=1}^{m}\sum_{i=1}^{n}g_{i}\frac{\partial (s_{j}{\circ}\vp)}{\partial x_{i}}\frac{\partial}{\partial s_{j}}
\in\  T^{0}_{S}(\calox)\ \subseteq\ \bigoplus_{j=1}^{m}\calox\frac{\partial}{\partial s_{j}}\,.
\label{al:vfformt0sox} 
\end{align}
\end{sit}

\begin{sit}
Of particular importance is the $\calos$--linear {\em Kodaira--Spencer map\/} defined by $\vp$. It is the 
composition
\begin{align*}
\delta_{KS}=\delta_{KS}^{\vp}=\delta\circ T^{0}_{S}(\vp^{\flat}):
T^{0}_{S}(\calos)\ \xto{T^{0}_{S}(\vp^{\flat})}\ T^{0}_{S}(\calox)\xto{\ \delta\ } T^{1}_{X/S}\,
\end{align*}
that sends a vector field $D=\sum_{j=1}^{m}f_{j}\frac{\partial}{\partial s_{j}}\in T^{0}_{S}$ to the
class 
\begin{align*}
\delta_{KS}(D)=\delta\Bigg(\sum_{j=1}^{m}f_{j}\vp\frac{\partial}{\partial
s_{j}}\Bigg)\in T^1_{X/S}.
\end{align*}
Thus we have a commutative diagram
\begin{align*}
\xymatrix{
  & & & T^0_S \ar[d]_{T^0_S(\vp^{\flat})} \ar[dr]^{\delta_{KS}} \\
0\ar[r]&T^{0}_{X/S}\ar[r]&T^{0}_{X}\ar[r]^-{\Jac(\vp)}&T^{0}_{S}(\calox)\ar[r]^{\delta}&T^{1}_{X/S}\ar[r]&
T^{1}_{X}\ar[r]
&\cdots\,.
}
\end{align*}
\end{sit}

\begin{sit}
The significance of the Kodaira--Spencer map is twofold: a vector field $D\in T^{0}_{S}$ is {\em liftable\/} 
to $X$, if, and only if, $\delta_{KS}(D)=0$. Indeed, the exactness of the Zariski--Jacobi tangent 
cohomology sequence shows that the image 
$T^{0}_{S}(\vp^\flat)(D)=\sum_{j=1}^{m}f_{j}\vp\frac{\partial}{\partial s_{j}}$ in $T^{0}_{S}(\calox)$ of the 
vector field $D$ is induced from a vector field $E$ on $X$, in that $T^{0}_{S}(\vp^\flat)(D)= \Jac(\vp)(E)$,
for some $E$ if, and only if, $\delta_{KS}(D)=0$. One therefore calls the kernel of the Kodaira--Spencer 
map also the $\calos$--submodule of {\em liftable vector fields\/} in $T^{0}_{S}$.

A deformation-theoretic interpretation is that such a lift trivializes the
infinitesimal first-order deformation of $X/S$ along $D$, whence we also say that $X/S$ is
(infinitesimally) {\em trivial along\/} $D$ as soon as $\delta_{KS}(D)=0$.
\end{sit}

\begin{sit}
On the other hand, if $\vp$ is a {\em flat} morphism, then the Kodaira--Spencer map is {\em surjective}, if, 
and only if, $\vp$ represents a {\em versal deformation\/} of the fibre $X_{0}=\vp^{-1}(0)\subseteq X$ of 
$\vp$ over the origin (\cite{Flenner}).
\end{sit}

\begin{sit}
If $X$ is smooth, then $T^1_X=0$ and
the inclusion
\eqref{al:vfformt0x} is an equality, while the inclusion
\eqref{al:vfformt0sox} becomes an equality if $S$ is smooth.

In particular, if both $X$ and $S$ are smooth, the dual Zariski--Jacobi sequence truncates to
a resolution of $T^1_{X/S}$, with $T^0_{X/S}$ as a second
syzygy module.

In the language of
the Thom--Mather theory of the singularities of differentiable maps,
$T^1_{X/S}$ is isomorphic as a vector space to the
\emph{extended normal space} of $\vp$ under
\emph{right equivalence}, while
the cokernel of $\delta_{KS}:T^0_S\to T^1_{X/S}$ is isomorphic
as a vector space
to the
extended normal space of $\vp$ under \emph{left-right equivalence}
(see \cite{greuel-le}).
\end{sit}

\subsection*{Free divisors}
After this short excursion into the general theory of the
\mbox{(co-)tangent} complex and its cohomology,
we recall the pertinent facts about free divisors.

\begin{sit}
\label{sit:definederlog}
Let $f\in\calos$ be the germ of a nonzero function on a smooth germ $S$ with zero locus the {\em  divisor}, or {\em hypersurface\/} 
germ $V(f)\equiv \{f=0\}\subseteq S$. Differentiating $f$ yields the commutative diagram 
in Figure \ref{fddiag} of 
$\calos$--modules with exact rows and exact columns, the rows exhibiting, one may say, defining, the 
{\em singular locus\/} $\Sigma$ of $V(f)$ as well as the $\calos$--module $\Der_{S}(-\log f)$ of 
{\em logarithmic vector fields\/} on $S$ along $V(f)$, as cokernel, respectively kernel, of the 
$\calos$--linear maps in the middle.
%
%
\begin{figure}
\begin{equation*}
\xymatrix{
&&0\ar[d]&&0\ar[d]\\
&&\calos\ar@{=}[rr]\ar[d]_{in_{2}}&&\calos\ar[d]^{\cdot f}\\
0\ar[r]&\Der_S(-\log f)\ar[r]^-{\tilde\sigma_{f}}\ar@{=}[d]
&T^{0}_{S}\oplus\calos\ar[rr]^-{(\Jac(f),f)}\ar[d]^{pr_{1}}
&&\calos\ar[r]\ar[d]&\calo_{\Sigma}\ar@{=}[d]\ar[r]&0&(\dagger)\\
0\ar[r]&\Der_S(-\log f)\ar[r]^-{\sigma_{f}}&T^{0}_{S}\ar[rr]^-{\jac(f)}\ar[d]&&\calo_S/(f)\ar[r]\ar[d]&\calo_{\Sigma} \ar[r]&0\\
&&0&&0
}
\end{equation*}
\caption{The commutative diagram exhibiting $\Sigma$ and $\Der_S(-\log
f)$, described in \ref{sit:definederlog}--\ref{sit:saitomatrix}.}
\label{fddiag}
\end{figure}

Here $\Jac(f)(D) = D(f)$ for any vector field or derivation $D\in T^{0}_{S}$, and $\jac(f)(D)$ is the class of
$D(f)$ modulo $f$.
\end{sit}

\begin{defn}
Recall that $f$ defines a free divisor in $S$, if $f$ is {\em reduced\/} and $\Der_{S}(-\log f)$ is a {\em free\/} 
$\calos$--module, necessarily of rank $m=\dim S$ (see \cite{Saito}).
\end{defn}

We recall the basic notions of the theory.

\begin{sit}
\label{sit:saitomatrix}
If $f\in\calos$ defines a free divisor, $\{e_{j}\}_{j=1}^{m}$ is a choice of an $\calos$--basis of 
$\Der_{S}(-\log f)$, and
$\{\partial/\partial s_{j}\}_{j=1}^m$
is the canonical basis of $T^{0}_{S}$ determined by local coordinates
$s_{j}$ on $S$ (as
in \ref{sit:ZJ}), then the matrix of the inclusion $\sigma_{f}$ in
Figure \ref{fddiag} with respect to these
bases is  
a {\em Saito\/} or {\em discriminant matrix\/} for $f$.

The matrix of $\tilde\sigma_{f}$, where we extend the basis
$\{\partial/\partial s_{j}\}_{j=1}^m$ of $T^{0}_{S}$ by 
the canonical basis $1\in \calos$ of that free $\calos$--module of rank $1$, yields then the 
{\em extended\/} Saito or discriminant matrix for $f$, in that 
\begin{align*}
\tilde\sigma_{f}(e_{j}) &= \left(\sum_{i=1}^{m}a_{ij}\frac{\partial}{\partial s_{i}}, -h_{j}\right)
\end{align*}
records that the vector field
$D_{j}=\sum_{i=1}^{m}a_{ij}\frac{\partial}{\partial
s_{i}}=\sigma_f(e_j)$
is logarithmic along $f$,
with
\begin{align*}
D_{j}(\log f):=\frac{D_{j}(f)}{f} = h_{j}\in\calos\,.
\end{align*}
The exactness at $T^0_S\oplus \calos$ of the row $(\dagger)$ in Figure \ref{fddiag}
follows from the observation that
\begin{align*}
(\Jac(f),f)(D,-h)=D(f)-f\cdot h
\end{align*}
vanishes if, and only if, $D$ is logarithmic with
$D(\log f)=h$.

Moreover, the minor $\Delta_{j}$ obtained by removing the column
of $\tilde\sigma_{f}$ corresponding to $\partial/\partial s_{j}$
and taking the determinant of the remaining square matrix equals $\partial f/\partial s_{j}$ up to multiplication by a unit, while the
determinant of the
matrix of $\sigma_{f}$ with respect to the chosen bases returns $f$ times a unit.

In these terms, the vector fields $D_{1},..., D_{m}$ form a basis of
the logarithmic vector fields as a submodule of 
$T^{0}_{S}$.
\end{sit}

The commutative diagram in Figure \ref{fddiag} yields as well Aleksandrov's characterization of free divisors.
\begin{proposition}[\cite{Aleksandrov}]
\label{prop:aleks}
A hypersurface germ $V(f)\subset S$ is a free divisor, if, and only
if, the singular\footnote{%
The empty set has any codimension.
} locus $\Sigma$ is 
Cohen--Macaulay of codimension $2$ in $S$.
\end{proposition}
\begin{proof}
Indeed, the codimension of $\Sigma$ in $S$ is at least $2$ if, and only if, $f$ is squarefree, that is, $V(f)$ is 
reduced. On the other hand, $\Der_{S}(-\log f)$ is free if, and only if, $\calo_{\Sigma}$ is of projective 
dimension, and thus, of codimension at most $2$.
\end{proof}

To prepare for our main result, we record how Figure \ref{fddiag}
behaves with respect to base change.
\begin{lemma}
\label{lem:pb}
Assume $V(f)\subset S$ is a free divisor and let $\vp:X\to S$ be a morphism from an analytic germ $X$ to $S$.
If $X$ is Cohen--Macaulay and the inverse image $\vp^{-1}(\Sigma)$ of the singular locus $\Sigma$ of $V(f)$
is still%
\footnote{Note that the codimension cannot go up under pullback.}
of codimension $2$, then the exact row $(\dagger)$ in
Figure \ref{fddiag} pulls back to an exact sequence
\begin{align*}
\xymatrix{
0\ar[r]&\vp^{*}\Der_S(-\log f)\ar[r]^-{\theta}
&T^{0}_{S}(\calox)\oplus\calox\ar[r]^-{\alpha}
&\calox\ar[r]&\calo_{\vp^{-1}(\Sigma)}\ar[r]&0\,,
}
\end{align*}
where
$\theta(1\otimes D)=(T^0_S(\vp^\flat)(\sigma_f(D)),
\vp^\flat(-\sigma_f(D)(\log f)))$
and
$\alpha = \vp^{\flat}(\Jac(f),f)=(\Jac(f)\vp,f\vp)$.

If furthermore $f\vp$ remains a non-zero-divisor in $\calox$, then 
the pull back of Figure \ref{fddiag} by $\vp$ gives a diagram with
exact rows and columns.
\end{lemma}
\begin{proof}
Let $I=I(\Sigma)\subset \calos$.

Apply $\calox\otimes_{\calos}\!\text{--}\,$ to the exact sequence
$(\dagger)$ in Figure \ref{fddiag} to get
the pullback 
$\calox$--complex $C$.
The last nonzero term of $C$ is
$\vp^*(\calo_{\Sigma})\cong \calox/J$ for $J=\calox\cdot\vp^\flat(I)$, and this
is the pullback of the scheme $\calo_{\Sigma}$, as claimed.
It is straightforward to check that under the identification
of the two middle free modules in $C$
with $T^0_S(\calox)\oplus \calox$ and $\calox$ respectively,
the complex $C$ is the sequence given in the statement.

For $i\geq 0$, we have 
$H_i(C)=\Tor_i^{\calos}(\calox,\calo_{\Sigma})$,
and so
each homology module is supported on $\vp^{-1}(\Sigma)$.
In particular,
$J$ annihilates each $H_i(C)$.

Since $X$ is Cohen--Macaulay,
$\depth(J,\calox)=\codim(J)=2$, and
similarly
$\depth(J,(\calox)^k)=2$ for $k\geq 1$.
Thus, if
$C_0=\calox,C_1,\ldots$ are the free modules in $C$,
then
$\depth(J,C_i)>i-1$ for all $i\geq 1$.
%
This fact, and the earlier observation that $J\cdot H_i(C)=0$ for
$i\geq 1$, are
enough to ensure that $H_i(C)=0$ for all $i\geq 1$
(see \cite[\S1.2, Corollaire 1]{bourbakicommalgX}).
Thus $C$ is exact.

The second assertion then follows.
\end{proof}


\section{The Main Results}
\label{sec:mainresult}
\begin{sit}
Fix as before a smooth germ $S$ and let $f\in\calos$ define a free divisor $V(f)\subset S$ with singular locus
$\Sigma\subset V(f)$.
By our definition, $f$ is reduced.
In fact, reducedness does not 
matter when computing the module of logarithmic vector fields for a
hypersurface.
\end{sit}

\begin{lemma}[%
{\cite[p. 313]{hausermuller}},
{\cite[Lemma 3.4]{GS}}]
\label{lemma:notreduced}
If 
$X$ is smooth and
$h_1,h_2\in \calox$ 
define the same zero loci as sets in $X$, then
$\Der_X(-\log h_1)=\Der_X(-\log h_2)$.
\end{lemma}
\begin{proof}
Let $g\in\calox$ factor
into distinct irreducible components
as
$g=g_1^{k_1}\cdots g_\ell^{k_\ell}$.
By an easy argument using the product rule and the fact that
$\calox$ is a unique factorization domain,
$\Der_X(-\log g)=\cap_{i} \Der_X(-\log g_i)$.
The result follows.
\end{proof}

\begin{example}
For $g=g_1^{k_1}\cdots g_\ell^{k_\ell}$ as in the proof, the logarithmic vector fields satisfy
$\Der_X(-\log g)=\Der_X(-\log g_1\cdots g_\ell)$.
\end{example}

We now give our main result, a sufficient condition for the reduction of $f\vp$ to
define a free divisor in $X$. 

\begin{theorem}
\label{altmthm}
Let $\vp:X\to S$ be a morphism of smooth germs and let $f\in\calos$ define a free divisor 
$V(f)\subset S$ with singular locus $\Sigma\subset V(f)$. 
Assume that $\mathrm{Image}(\vp)\nsubseteq V(f)$, i.e., $f\vp$ is not zero.
Let $g$ be a \emph{reduction} of $f\vp$ in $\calox$, a reduced function defining
the same zero locus as $f\vp$.
If 
\begin{enumerate}[\rm(a)]
\item
\label{mthm.t0xsiscm} 
the module of vertical vector fields
$T^0_{X/S}$ is free,
\item 
\label{mthm.ks}
the Kodaira--Spencer map $\delta_{KS}:T^{0}_{S}\to T^{1}_{X/S}$ vanishes on $\Der_{S}(-\log f)$, that is,
$\delta_{KS}\circ \sigma_{f} = 0$, and
\item
\label{mthm.singloc}
the inverse image $\vp^{-1}(\Sigma)$ of the singular locus is still of codimension $2$ in $X$,
\end{enumerate}
then $g$ defines a {\em free divisor\/} in $X$ and its
$\calox$--module of logarithmic vector fields satisfies
\begin{align}
\label{al:altderxeqn}
\Der_X(-\log g)=
\Der_{X}(-\log f\vp)\cong T^{0}_{X/S}\oplus\vp^{*}\Der_{S}(-\log f)\,.
\end{align}
\end{theorem}

If $\Sigma=\emptyset$, then 
by our convention on the codimension of the empty set,
\eqref{mthm.singloc} is satisfied.

\begin{proof}
The three $\calox$--linear maps:
\begin{itemize}
\item 
$\alpha =  \vp^{\flat}(\Jac(f),f): T^{0}_{S}(\calox)\oplus \calox\to \calox$ as above,
\item
$\beta = \Jac(\vp)\oplus \id_{\calox}: T^{0}_{X}\oplus \calox\to T^{0}_{S}(\calox)\oplus\calox$, and
\item
$\gamma=(\Jac(f\vp),f\vp):T^{0}_{X}\oplus\calox \to \calox$
\end{itemize}
satisfy $\gamma=\alpha\beta$ and give rise to the following diagram relating kernels and cokernels 
of these maps, where $I$ is the ideal generated by
$f\vp$ and its partial derivatives.
\begin{align*}
\xymatrix{
&&&0\ar[d]&&0\\
&&&\vp^{*}\Der_{S}(-\log
f)\ar[d]^-{\theta}\ar[r]^-{\omega}&T^{1}_{X/S}\ar[dd]\ar[ru]\\
&&&T^{0}_{S}(\calox)\oplus\calox\ar[d]^{\alpha}\ar[ru]_-{\delta+0}\\
0\ar[r]&\Der_{X}(-\log f\vp)\ar[r]^-{\tilde\sigma_{f\vp}}\ar[rruu]&
T^{0}_{X}\oplus \calox\ar[ru]^{\beta}\ar[r]^-{\gamma}\ar[r]&
\calox\ar[r]\ar[d]&\calox/I\ar[r]\ar[ld]&0\\
&T^{0}_{X/S}\ar[u]\ar[ru]&&\calo_{\vp^{-1}(\Sigma)}\ar[d]\ar[ld]\\
0\ar[ru]&0\ar[u]&0&0
}
\end{align*}

The horizontal exact sequence involving $\gamma$
is the one described in \ref{sit:definederlog} and \ref{sit:saitomatrix}
that defines 
$\Der_X(-\log f\vp)$
and 
the singular locus
of $V(f\vp)$ as a scheme.
The vertical sequence involving $\alpha$ is exact by 
\eqref{mthm.singloc} as explained in 
Lemma \ref{lem:pb} above.
As $X$ is smooth,  $T^{1}_{X}=0$, and the diagonal exact sequence including $\beta$ is the direct sum of
the identity on $\calox$ and (the initial segment of) the exact dual Zariski--Jacobi sequence for $\vp$ as
recalled in \ref{sit:ZJ} above.
The remaining arrows form the Ker--Coker exact sequence defined
by $\gamma=\alpha\beta$.

The construction of this exact sequence
shows that $\omega$, the $\calox$--linear map connecting
$\ker(\alpha)$ to $\coker(\beta)$, satisfies
$\omega=(\delta+0)\circ \theta$.
Hence if 
$D\in \Der_{S}(-\log f)\subseteq T^{0}_{S}$ and
$1\otimes D\in \calox\otimes_{\calos}  \Der_{S}(-\log f)$ is the
pulled-back vector field in $\vp^{*}\Der_{S}(-\log f)$,
then
$\omega(1\otimes D)=
((\delta+0)\circ \theta)(1\otimes D)=
\delta(T^0_S(\vp^\flat)(D))=
\delta_{KS}(D)$.
Our assumption 
\eqref{mthm.ks}
is hence equivalent to $\omega=0$.
Thus,
the Ker--Coker exact sequence defined by $\gamma=\alpha\beta$ splits
into two short exact sequences for the kernels, respectively cokernels,
\begin{align}
\label{eqn:firstseq}
\xymatrix{
0\ar[r]&T^{0}_{X/S}\ar[r]&\Der_{X}(-\log f\vp)\ar[r]&\vp^{*}\Der_{S}(-\log f)\ar[r]&0}
\end{align}
and
\begin{align}
\label{eqn:secondseq}
\xymatrix{
0\ar[r]&T^{1}_{X/S}\ar[r]&\calox/I\ar[r]&\calo_{\vp^{-1}(\Sigma)}\ar[r]&0}.
\end{align}
Since
$\Der_{S}(-\log f)$ is a free $\calos$--module
by assumption, and thus $\vp^{*}\Der_{S}(-\log f)$ is a free
$\calox$--module, it follows that \eqref{eqn:firstseq} splits,
giving
the decomposition of $\Der_X(-\log f\vp)$ in 
\eqref{al:altderxeqn}.
By \eqref{mthm.t0xsiscm}
and Lemma \ref{lemma:notreduced},
$\Der_X(-\log g)=\Der_X(-\log f\vp)$ is a free $\calox$--module
and hence $g$ defines a free divisor.
The sequence \eqref{eqn:secondseq} will be used in the sequel.
\end{proof}

As condition \eqref{mthm.t0xsiscm}
of Theorem \ref{altmthm}
can be difficult to prove directly,
it is often easier to verify the following stronger
hypotheses;
of all our examples, only Example \ref{ex:square} applies Theorem \ref{altmthm}. 

\begin{theorem}
\label{mthm}
Let $\vp:X\to S$ be a morphism of smooth germs and let $f\in\calos$ define a free divisor 
$V(f)\subset S$ with singular locus $\Sigma\subset V(f)$. 
If both 
\begin{enumerate}[\rm(a)]\setcounter{enumi}{1}
\item 
the Kodaira--Spencer map $\delta_{KS}:T^{0}_{S}\to T^{1}_{X/S}$ vanishes on $\Der_{S}(-\log f)$, that is,
$\delta_{KS}\circ \sigma_{f} = 0$, and
\setcounter{enumi}{3}
\item
\label{mthm.t1xs}
$T^1_{X/S}$ is Cohen--Macaulay of codimension $2$,
\end{enumerate}
then $f\vp$ is reduced and defines a free divisor,
and $\Der_X(-\log f\vp)$ has the decomposition as in 
\eqref{al:altderxeqn} of Theorem \ref{altmthm}.
\end{theorem}
\begin{proof}
We check the conditions of Theorem \ref{altmthm}.
Condition \eqref{mthm.ks} is assumed.

By \eqref{mthm.t1xs} and the Auslander--Buchsbaum formula,
$\projdim(T^1_{X/S})=2$.
This implies
\eqref{mthm.t0xsiscm}, that $T^0_{X/S}$ is free,
as it is a second
syzygy module of $T^1_{X/S}$ via the dual Zariski--Jacobi sequence
for $\vp$.

Since $T^1_{X/S}$ is supported on the critical locus $C(\vp)$ of the map,
$\vp$ is smooth off a set of codimension $2$.
In particular, this implies that $f\vp$ is nonzero:
if $f\vp=0$, so $\mathrm{Image}(\vp)\subseteq V(f)$, then $\vp$ is
nowhere smooth.

For \eqref{mthm.singloc},
first note that the codimension of $\vp^{-1}(\Sigma)$ is $\leq 2$,
as the codimension cannot go up under pullback. 
Let $\vp'$ and $\vp''$ be the restriction of $\vp$ to $C(\vp)$ and its
complement in $X$.  Then
$\vp^{-1}(\Sigma)=(\vp')^{-1}(\Sigma)\cup (\vp'')^{-1}(\Sigma)$, both
of which have codimension $\geq 2$ in $X$:
the first is contained in $C(\vp)$,
and the second because $\Sigma$ has 
codimension $2$ in $S$ and $\vp''$ is smooth.
Thus we have \eqref{mthm.singloc}.

By Theorem \ref{altmthm} and its proof, 
$\Der_X(-\log f\vp)$ is free,
with the decomposition as in \eqref{al:altderxeqn}
and
the exact sequence
from \eqref{eqn:secondseq},
$$
\xymatrix{
0 \ar[r] & T^1_{X/S} \ar[r] & \calox/I \ar[r] &
\calo_{\vp^{-1}(\Sigma)} \ar[r] & 0
},$$
where $I$ is generated by $f\vp$ and its partial derivatives,
so that $\calox/I=\calo_{\Sing(V(f\vp))}$.
The outer terms $T^1_{X/S}$ and $\calo_{\vp^{-1}(\Sigma)}$ are
Cohen--Macaulay $\calox$--modules of codimension $2$ by
assumption \eqref{mthm.t1xs} and \eqref{mthm.singloc}, whence
$\calox/I$ is also a Cohen--Macaulay $\calox$--module of codimension
$2$.
Since 
$\codim(\calox/I)=2$, 
$f\vp$ is necessarily reduced and hence defines a free divisor.
\end{proof}

\begin{remark}
If Theorem \ref{mthm} applies to some $\vp$ and $f$ with
$S\cong \CC^2$,
then the Theorem easily produces many examples of free divisors in $X$.
Namely, 
whenever
$fg\in\calos$ is reduced and nonzero, then
$(fg)\vp$ defines a free divisor on $X$,
because 
any reduced plane curve in $S$ is a free divisor, and
condition \eqref{mthm.ks}
follows from
$\Der_S(-\log fg)\subseteq \Der_S(-\log f)$.
\end{remark}

As a first application we obtain a result originally observed by Mond and van Straten \cite[Remark 1.5]{MvS}.
\begin{theorem}
\label{fibres}
Let $C$ be the germ of an isolated complete intersection curve singularity. If $\vp:X\to S$ is any 
versal deformation of $C$, then the union of the singular fibres of
$\vp$, that is, the pullback along $\vp$ 
of the discriminant $\Delta\subset S$ in the base, is a free divisor.

More generally, if $f=0$ defines a free divisor in $S$ that contains the discriminant as a component, 
then its pre-image $f\circ \vp=0$ defines a free divisor in $X$.
\end{theorem}

\begin{proof}
It is well known (see \cite[6.13, 6.12]{Looijenga}) that $\Delta$ is a free divisor in a smooth germ $S$, that $X$ is smooth
as well, and that $T^{1}_{X/S}$ is 
a Cohen--Macaulay $\calox$--module of codimension two. Finally, in this case the kernel of the
Kodaira--Spencer map $\delta_{KS}:T^{0}_{S}\to T^{1}_{X/S}$ consists
precisely of the logarithmic vector fields along $\Delta$ (see
\cite[6.14]{Looijenga}) and so all the assumptions of
Theorem \ref{mthm} are satisfied
for $\Delta$ itself and then also for any free divisor in $S$ that contains $\Delta$ as a component.
\end{proof}

\begin{example}
A versal deformation of the plane curve
defined by $x_1^3+x_2^2$
is the map
$\vp:X=\CC^3\to S=\CC^2$
defined by
$\vp(x_1,x_2,s_1)=(s_1,x_1^3+x_2^2+s_1 x_1)$.
With coordinates $(s_1,s_2)$ on $S$,
the discriminant of $\vp$ is the free divisor
defined by
$\Delta=4s_1^3+27s_2^2$.
The module of liftable vector fields is
$\Der_S(-\log \Delta)$.
%
%
%
%
By Theorem \ref{fibres},
$$\Delta\vp=27x_1^6+54x_1^3 x_2^2+54 x_1^4 s_1+27 x_2^4+54 x_1 x_2^2 s_1+27 x_1^2 s_1^2+4 s_1^3$$
defines a free divisor on $X$, and the same is true for
the lift of any reduced plane curve containing $\Delta$ as a component,
for example $(\Delta\cdot s_2)\vp$.
Note that $\Delta\vp$ is equivalent to the classical swallowtail.
%
%
\end{example}

\begin{remark}
Note that Theorem \ref{fibres} can only hold for versal deformations of isolated complete intersection singularities on
{\em curves}. Indeed, for a versal deformation of any isolated complete intersection singularity
the corresponding module $T^{1}_{X/S}$ is Cohen--Macaulay, but of codimension equal to the
dimension of the singularity plus one (\cite[6.12]{Looijenga}).
\end{remark}

%

\section{Lifting Euler vector fields}
\label{sec:lifteuler}
Theorem \ref{mthm} requires that
all elements of $\Der_S(-\log f)$ lift.
This hypothesis may be relaxed,
at least 
for the Euler vector field of a weighted-homogeneous free divisor.
We first examine how general Theorem \ref{mthm} is in a
well-understood situation.

\begin{example}
\label{ex:fdalready}
Suppose that $\vp:X=\CC^n\to S=\CC$
(and hence $f\circ \vp$ for $f=s_1$) already defines a free divisor.
What is the content of Theorem \ref{mthm} in this case?

Here, $T^1_{X/S}\cong \coker \Jac(\vp)\cong \calox/J_\vp$, where
$J_\vp$ is the Jacobian ideal generated by the partial derivatives of $\vp$.
If $\vp\in J_\vp$,
equivalently, there exists an ``Euler-like'' vector field $\eta$ such that
$\eta(\vp)=\vp$,
then $T^1_{X/S}$ is Cohen--Macaulay of codimension $2$ by
Proposition \ref{prop:aleks} as $\vp$ defines a free divisor.
Moreover,
the vector field $s_1\frac{\partial}{\partial s_1}$
that generates $\Der_S(-\log s_1)$ lifts
if and only if $\vp\in J_\vp$.
Hence, the hypotheses of Theorem \ref{mthm} are satisfied exactly when 
$\vp\in J_\vp$, in which
case the conclusion says that
$\Der_X(-\log \vp)$ is the direct sum of $\calox\cdot \eta$ and the
(vertical) vector fields that 
annihilate $\vp$. 
%
\end{example}

A free divisor without an Euler-like vector field does not have this
direct sum decomposition.
Hence,
as this Example suggests,
we may weaken the lifting condition of Theorem \ref{mthm},
modify the algebraic condition,
and obtain a conclusion that lacks 
the direct sum decomposition as in
\eqref{al:altderxeqn} of Theorem \ref{altmthm}.

\begin{corollary}
\label{cor:mthm}
Let $\vp:X\to S$ be a morphism of smooth germs
with module $L=\ker(\delta_{KS})\subseteq T^0_S$ of liftable vector fields.
Let $f\in \calos$ define
a free divisor with singular locus $\Sigma\subset V(f)$.
Let $(w_1,\ldots,w_m)$ be a set of nonnegative integral weights for
the coordinates $(s_1,\ldots,s_m)$ on $S$.
Let $E=\sum_{i=1}^m w_i s_i\frac{\partial}{\partial s_i}\in T^0_S$ be the
corresponding Euler vector field,
so that $T^0_S(\vp^\flat)(E)=\sum_{i=1}^m w_i (s_i\circ\vp)\frac{\partial}{\partial
s_i}\in T^0_S(\calox)$.
If $f$ is weighted homogeneous of degree $d$ with respect to these weights, if
\begin{equation}
\label{eqn:cormthmn}
N=T^0_S(\calox)/(\mathrm{Image}(\Jac(\vp))+\calox\cdot
T^0_S(\vp^\flat)(E))
\end{equation}
is a Cohen--Macaulay $\calox$--module of codimension $2$,
and if
$\Der_S(-\log f)\subseteq L+\calos\cdot E$,
then $f\circ\vp$ defines a free divisor.
\end{corollary}
%
%
%
%
\begin{proof}
Let $t$ be a coordinate on $\CC$, and let $\vp=(\vp_1,\ldots,\vp_m)$.
Define  
$\theta:Y=X\times \CC\to S$ by
$\theta(x,t)=(e^{w_1 t}\cdot \vp_1(x),\ldots,e^{w_m t}\cdot \vp_m(x))$.
Since
\begin{align*}
\theta^\flat(f)(x,t)
&=f(e^{w_1 t}\cdot \vp_1(x),\ldots,e^{w_m t}\cdot \vp_m(x)) \\
&=e^{d t}\cdot \vp^\flat(f)(x),
\end{align*}
and $e^{d t}$ is a unit in $\calo_Y$, if Theorem \ref{mthm} applies to
$\theta$ and $f$,
then the lift of $f$ via
$\theta$ will give a free divisor $V(f\circ \vp)\times \CC$ in $Y$.
It follows that $f\circ \vp$ defines a free divisor in $X$.
It remains only to check the hypotheses of the Theorem.

A matrix representation of $\Jac(\theta)$ is
\begin{equation}
\label{eqn:jactheta}
\begin{pmatrix}
e^{w_1 t}\frac{\partial \vp_1}{\partial x_1} &\cdots
  & e^{w_1 t}\frac{\partial \vp_1}{\partial x_n}& w_1 e^{w_1 t}\vp_1 \\
\vdots & \ddots & \vdots & \vdots \\
e^{w_m t}\frac{\partial \vp_m}{\partial x_1} &\cdots
  & e^{w_m t}\frac{\partial \vp_m}{\partial x_n}& w_m e^{w_m t}\vp_m
\end{pmatrix},
\end{equation}
with values in $T^0_S(\caloy)$.
The isomorphism $\psi:T^0_S(\caloy)\to T^0_S(\caloy)$ with $\psi(\frac{\partial}{\partial s_i})=e^{-w_i t}\frac{\partial}{\partial s_i}$
shows that
deleting the exponential coefficients in \eqref{eqn:jactheta} gives an
isomorphic cokernel.
Thus, 
$T^1_{Y/S}\cong
\coker \Jac(\theta)$ is isomorphic
to $N\otimes_{\calox} \caloy$, 
and hence a
Cohen--Macaulay $\caloy$--module of codimension $2$.
This establishes condition \eqref{mthm.t1xs} of Theorem \ref{mthm}.

Now 
let $\eta=\sum_{i=1}^m a_i\frac{\partial}{\partial s_i}\in T^0_S$ be homogeneous
of degree $\lambda$, in that $\lambda=\deg(a_i)-w_i$ for $i=1,\ldots,m$.
Suppose that $\eta$ lifts under $\vp$ to
some $\xi=\sum_{j=1}^n b_j\frac{\partial}{\partial x_j}\in T^0_X$,
so that
$a_i\circ \vp=\sum_{j=1}^n b_j \frac{\partial \vp_i}{\partial
x_j}$ for $i=1,\ldots,m$.
Let $\xi'\in T^0_Y$ have the same defining equation.
Then
\begin{align*}
\Jac(\theta)\left( e^{\lambda t}\cdot \xi'\right)
&=
\sum_{i=1}^m e^{(\lambda+w_i)t}
  \left( \sum_{j=1}^n b_j \frac{\partial \vp_i}{\partial x_j} \right)
  \frac{\partial}{\partial s_i} \\
&=
\sum_{i=1}^m e^{\deg(a_i)t}\cdot
  (a_i\circ \vp)  
  \frac{\partial}{\partial s_i} \\
&=
\sum_{i=1}^m
a_i\circ \left(
  e^{w_1 t}\cdot \vp_1,
  \ldots,
  e^{w_m t}\cdot \vp_m\right) \frac{\partial}{\partial s_i}
\\
&=
T^0_S(\theta^\flat)(\eta).
\end{align*}
Thus, homogeneous elements of $L$ lift via $\theta$.
The Euler vector field
$E$ also lifts, as
$\Jac(\theta)(\frac{\partial}{\partial t})
=T^0_S(\theta^\flat)(E)$.
It follows that the module generated by homogeneous elements of
$L+\calos \cdot E$ lifts via $\theta$. 
Since $f$ is weighted homogeneous, $\Der_S(-\log f)$ has a homogeneous
generating set
and hence elements of $\Der_S(-\log f)$ lift via $\theta$,
verifying condition \eqref{mthm.ks} of Theorem \ref{mthm}.
\end{proof}

\begin{example}
When $S=\CC$ and $f=s_1$, as in Example \ref{ex:fdalready},
then
$\Der_S(-\log f)=\calos\cdot E$
and $N\cong \calox/(J_\vp,\vp)$.
In this case, 
Corollary \ref{cor:mthm} reduces to
the ``if'' direction of 
Proposition \ref{prop:aleks}.
\end{example} 

\begin{sit}
This corollary 
may create free divisors without an Euler--like vector field,
and 
may be applied to maps between spaces of the same
dimension.
\end{sit}

%
%
\begin{example}
%
%
%
Let $\vp:X=\CC^3\to S=\CC^2$ be defined by
$\vp(x_1,x_2,x_3)=(x_1^2+x_2^3,x_2^2+x_1x_3)$,
and let $f=s_1s_2(s_1+s_2)$.
Let $L$ be the module of vector fields liftable through $\vp$.
Although $T^1_{X/S}$ is Cohen--Macaulay of codimension $2$,
$\Der_S(-\log f)\nsubseteq L$.
For weights $w_1=w_2=1$,
we have
$\Der_S(-\log f)\subseteq L+\calos\cdot E$,
and
the module of \eqref{eqn:cormthmn} is also Cohen--Macaulay of
codimension $2$. 
By Corollary \ref{cor:mthm},
\begin{align*}
f\circ \vp = (x_1^2+x_2^3)(x_2^2+x_1x_3)(x_1^2+x_2^3 + x_2^2+x_1x_3)
\end{align*}
 defines a free divisor;
it has no Euler-like vector field.
\end{example}

\begin{example}
Let $\vp:X=\CC^3\to S=\CC^3$ be defined by
$\vp(x_1,x_2,x_3)=(x_1x_3+x_2^2,x_2,x_3)$.
For $w_1=w_2=w_3=1$ the module of \eqref{eqn:cormthmn}
is Cohen--Macaulay of codimension $2$, although $T^1_{X/S}$
is not.
As $L+\calos \cdot E$ contains $\Der_S(-\log f)$ for, e.g.,
$f=s_1s_2s_3$ or $f=s_1s_3(s_1s_3-s_2^2)$,
by Corollary \ref{cor:mthm}
each such $f\circ \vp$ defines a free divisor in $X$.
\end{example}

\begin{remark}
If $f$ is multi-weighted homogeneous, that is, weighted homogeneous of
degree $d_k$ with respect to weights $(w_{1k},\ldots,w_{mk})$ for
$k=1,\ldots,p$
(or, $f=0$ is invariant under the action of an algebraic $p$-torus),
then  
a version of Corollary \ref{cor:mthm} holds, with 
$E$ replaced by the $p$ Euler vector fields. 
To adapt the proof, let
$\theta:X\times \CC^p\to S$ be defined by
$$\theta(x,t)=\left(
  e^{\sum_{k=1}^p w_{1k} t_k} \cdot \vp_1(x),
\cdots,
  e^{\sum_{k=1}^p w_{mk} t_k} \cdot \vp_m(x)
 \right),$$
for $\vp=(\vp_1,\ldots,\vp_m)$,
and then consider multi-weighted homogeneous
vector fields.

For instance, if $f=s_1\cdots s_m$ is the normal crossings divisor
in $S=\CC^m$
with $m$ weightings of the form $(0,\cdots,1,\cdots,0)$,
then the Euler vector fields generate $\Der_S(-\log f)$, hence
the liftability condition is satisfied for any $\vp$.
The analog of the module $N$ of \eqref{eqn:cormthmn}
is the cokernel $N'$ of
$$
A=\begin{pmatrix}
\frac{\partial \vp_1}{\partial x_1} &\cdots & \frac{\partial \vp_1}{\partial x_n}& \vp_1 & 0 & \cdots & 0 \\
\frac{\partial \vp_2}{\partial x_1} &\cdots & \frac{\partial
\vp_2}{\partial x_n}& 0 & \vp_2 & \cdots & 0 \\
\vdots & \ddots & \vdots & \vdots & \vdots & \ddots & \vdots \\
\frac{\partial \vp_m}{\partial x_1} &\cdots & \frac{\partial
\vp_m}{\partial x_n}& 0 & 0 & \cdots & \vp_m
\end{pmatrix}.
$$
When each $\vp_i$ is nonzero,
note that
$$\ker(A)
\cong \cap_{i} \Der_X(-\log \vp_i)
=\Der_X(-\log \vp_1\cdots \vp_m)
=\Der_X(-\log \vp^\flat(f)).$$
Then $N'$ is Cohen--Macaulay of codimension 2 if and only if
$\projdim(N')=2$ and $\dim(N')=\dim(X)-2$; although the
former is enough to ensure $\ker(A)$ is free,
the latter ensures $\vp_1\cdots \vp_m$
is reduced and nonzero.
\end{remark}

\begin{remark}
A result similar to Corollary \ref{cor:mthm}
may be obtained by 
applying Theorem
\ref{altmthm} instead of Theorem \ref{mthm}.
\end{remark}

\section{Adding components and dimensions}
\label{sec:addcompdim}
We now examine a way to add components to a free divisor on
$\CC^m$ to produce a free divisor
on $\CC^m\times \CC^n$.
Use coordinates $(x_1,\ldots,x_m)$ and $(y_1,\ldots,y_n)$ on $\CC^m$ and $\CC^n$
respectively.

For an $\calox$--ideal $I$ on a smooth germ $X$, define
the $\calox$--module of logarithmic vector fields by
$$\Der_X(-\log I)=\{\eta\in \Der_X: \eta(I)\subseteq I\}.$$
This agrees with our earlier definition for hypersurfaces.

\begin{proposition}
\label{prop:ffstar}
Let $I=(g_1,\ldots,g_n)$ be a $\calo_{\CC^m}$--ideal such that
$\calo_{\CC^m}/I$ is Cohen--Macaulay of codimension 2.
If $h\in\calo_{\CC^m}$ defines a free divisor on $\CC^m$
with 
\begin{equation}
\label{eqn:containderlog}
\Der_{\CC^m}(-\log h)\subseteq \Der_{\CC^m}(-\log
I),
\end{equation}
then $h\cdot \left(\sum_{i=1}^n g_iy_i\right)$ defines a free divisor on
$X=\CC^m\times \CC^n$.
\end{proposition}
\begin{proof}
Let $S=\CC^m\times \CC$ have coordinates $(z_1,\ldots,z_m,t)$
and view $g_i$ and $h$ as elements of $\calos$.
Define $\varphi:X\to S$ by $\varphi(x,y)=(x,\sum_{i=1}^n g_i(x)\cdot y_i)$.
Let $f(z,t)=h(z)\cdot t$ define the free divisor in $S$ which is the
``product-union'' of $V(h)\subset \CC^m$ and $\{0\}\subset \CC$.
The statement will then follow from Theorem \ref{mthm} by lifting
$f$ via $\varphi$.

To check condition \eqref{mthm.t1xs} of the Theorem, observe that with
respect to the coordinates given, the
matrix form of the Jacobian is 
$$
\Jac(\varphi)=
\begin{pmatrix}
{\textrm{\Large I}}_{m,m} &\phantom{M} & &{\textrm{\Large 0}}_{m,n} & \\
*   &  & g_1 & \cdots & g_n
\end{pmatrix},$$
where the subscripts on $I$ and $0$ denote the sizes of identity
and zero blocks respectively.
In particular,
$T^1_{X/S}\cong \coker \Jac(\varphi)$ is isomorphic to
$\calox/(I\otimes_{\calo_{\CC^m}} \calox)
\cong (\calo_{\CC^m}/I)\otimes_{\calo_{\CC^m}} \calox$.
Since $\calo_{\CC^m}/I$ is a Cohen--Macaulay $\calo_{\CC^m}$--module of
codimension 2, by flatness of $\calox$ over $\calo_{\CC^m}$ it follows
that $T^1_{X/S}$ is a Cohen--Macaulay $\calox$--module of the same codimension.

For \eqref{mthm.ks}, $\Der(-\log f)$ is generated by elements of
$\Der(-\log h)$ extended to $S$ with $0$ as the coefficient of
$\frac{\partial}{\partial t}$, together with $t\frac{\partial}{\partial t}$.
The latter lifts:
$$
\Jac(\varphi)\left(\sum_{i=1}^n y_i \frac{\partial}{\partial
y_i}\right)
=\left(\sum_{i=1}^n g_i y_i\right)\frac{\partial}{\partial t}
=T^0_S(\vp^\flat)\left(t\frac{\partial}{\partial t}\right).
$$
Now, 
if $\eta=\sum_{i=1}^m a_i \frac{\partial}{\partial z_i}\in\Der_{\CC^m}$ is
logarithmic for $I$, then
there exist $\gamma_{j,k}\in\calo_{\CC^m}$ such that
$\eta(g_j)=\sum_{k=1}^n \gamma_{j,k}\cdot g_k$ for all $j$. 
Then $\eta$ extended to $S$ lifts as well:
\begin{align*}
\Jac(\varphi)&\left(
\sum_{i=1}^m a_i \frac{\partial}{\partial x_i}
-\sum_{j,k=1}^n \gamma_{j,k}y_j\frac{\partial}{\partial y_k}
\right)
\\
&= 
\sum_{i=1}^m a_i \left(\frac{\partial}{\partial z_i}
  +\left(\sum_{j=1}^n \frac{\partial g_j}{\partial z_i} y_j \right)
  \frac{\partial}{\partial t} \right)
-
\left(\sum_{j,k=1}^n \gamma_{j,k} g_k y_j\right)\frac{\partial}{\partial t}
\\
&= 
\sum_{i=1}^m a_i \frac{\partial}{\partial z_i}
+
\left(\sum_{i=1}^m \sum_{j=1}^n a_i\frac{\partial g_j}{\partial z_i} y_j
\right)
\frac{\partial}{\partial t}
-
\left(\sum_{j=1}^n \eta(g_j) y_j \right)\frac{\partial}{\partial t}
\\
&=
\sum_{i=1}^m a_i \frac{\partial}{\partial z_i}
+
\left(
\sum_{j=1}^n \eta(g_j) y_j \right)
\frac{\partial}{\partial t}
-
\left(\sum_{j=1}^n \eta(g_j) y_j \right)\frac{\partial}{\partial t}
\\
&
=
T^0_S(\vp^\flat)(\eta).
\end{align*}
In view of assumption \eqref{eqn:containderlog}, thus all
generators of $\Der(-\log f)$ lift.
\end{proof}

\begin{remark}
By the form of $\Jac(\vp)$ in the proof,
the (free module of) vertical vector fields of $\vp$ are
generated by the $\calo_{\CC^m}$--syzygies of $\{g_1,\ldots,g_n\}$.
\end{remark}

\begin{remark}
There is no need for $(g_1,\ldots,g_n)$ to be a minimal
generating set.
\end{remark}

\begin{remark}
\label{rem:smoothffstar}
The conclusion of Proposition \ref{prop:ffstar} also holds if
$I=(1)$.
Then some $g_i$ is a unit in the local ring, and 
so a local change of coordinates of $X$ takes
$h\cdot \left(\sum_{i=1}^n g_i y_i\right)$ to 
$h\cdot y_1$, which defines a ``product-union'' of free divisors.
\end{remark}

\begin{sit}
To find an $h$ and $I$ that satisfy assumption \eqref{eqn:containderlog},
a natural approach is to use the ideal
$(J_h,h)$ defining the singular locus 
$\Sigma$
of $V(h)$.
In particular, we have the following generalization of 
the ``$ff^*$''
construction of Buchweitz--Conca (\cite[Theorem 8.1]{BC}), where we have
removed the hypothesis that $h$ be weighted homogeneous.
\end{sit}

\begin{corollary}
\label{cor:ffstar}
If $h\in\calo_{\CC^m}$ defines a
free divisor on $\CC^m$ and
$g_1,\ldots,g_n$ generate the $\calo_{\CC^m}$--ideal $I=(J_h,h)$, then
$h\cdot \left(\sum_{i=1}^n g_i y_i \right)$
defines a free divisor on $\CC^m\times \CC^n$.
In particular,
$h\cdot \left(hy_{m+1} +\sum_{i=1}^m \frac{\partial h}{\partial x_i} y_i \right)$
always defines a free divisor on $\CC^m\times \CC^{m+1}$,
and if $h\in J_h$ then 
$h\cdot \left(\sum_{i=1}^m \frac{\partial h}{\partial x_i} y_i \right)$
defines a free divisor on $\CC^m\times \CC^m$.
\end{corollary}
\begin{proof}
It is enough to prove the first assertion, as the rest follows from it.
Let $\Sigma$ be the singular locus of $V(h)$, defined by $I$.
If $V(h)$ is smooth, then $I=(1)$ and we may apply Remark
\ref{rem:smoothffstar}.
Otherwise, $\calo_{\CC^m}/I$ is Cohen--Macaulay of codimension $2$ by
Proposition \ref{prop:aleks}.
Any vector field that is logarithmic to $V(h)$ is also logarithmic to
$I$, as is easily seen from the product rule.
Now apply Proposition \ref{prop:ffstar}.
\end{proof}

However, this is not the only way to find a satisfactory $h$ and $I$.

\begin{example}
Let $M=M_{n,n}$ be the space of $n\times n$ complex matrices with coordinates $\{x_{ij}\}$,
let $N=M_{n-1,n}$,
and let
$\pi:M\to N$ be the projection that deletes the last row.
Differentiate 
$\rho:\GL({n-1},\CC)\times\GL({n},\CC)\to\GL(N)$
defined by $\rho(A,B)(X)=AXB^{-1}$
to obtain a finite-dimensional Lie algebra 
$\fg$ of linear vector fields on $N$.
Let $D\subset \Der_N$ be the $\mathcal O_{N}$--submodule generated
by $\fg$.
Let $f$ define a free divisor on $N$ for which
$\Der_N(-\log f)\subseteq D$;
for instance, $f$ could be a linear free divisor on $N$ obtained by
restricting $\rho$ to an appropriate subgroup.

Now $\rho$ leaves invariant 
$N_0=\{X: \rank(X)<n-1\}\subset N$, and hence
all elements of $\fg$, $D$, and $\Der_N(\log f)$ are tangent to the
variety $N_0$.
Note that
$N_0$ is Cohen--Macaulay of codimension $2$ and defined by
$I=((-1)^{n+1} g_1,\ldots,(-1)^{n+n} g_n)$, where 
$g_i:N\to\CC$ deletes column $i$ and takes the determinant.
Since
$\sum_{i=1}^n (-1)^{n+i} g_i x_{ni}=\det$ on $M$,
by Proposition \ref{prop:ffstar}, 
$(f\circ \pi)\cdot \det$ defines a free divisor on $M$.
By the lifts in the proof and the observation that
the vertical vector fields are generated by linear vector fields
(e.g., by Hilbert--Burch),
we see that if $f$ defines a linear free divisor on $N$ then 
$(f\circ \pi)\cdot \det$ defines a linear free divisor on $M$.
(This linear free divisor case partially recovers
\cite[Prop.~5.3.7]{Pike}.)

As a concrete example, 
for the linear free divisor on $M_{2,3}$ defined by
$$f=x_{11}x_{12}
\begin{vmatrix} x_{11} & x_{12} \\ x_{21} & x_{22} \end{vmatrix}
\begin{vmatrix} x_{12} & x_{13} \\ x_{22} & x_{23} \end{vmatrix},$$
$(f\circ \pi)\cdot \det$ defines a linear free divisor
$D$ on $M_{3,3}$,
part of the ``modified LU'' series
of \cite[Theorem 7.1]{DP} or \cite[\S5.1]{Pike}.
In fact, $D$ may be constructed from $\{x_{11}=0\}\subset M_{1,1}$
by repeatedly applying Proposition \ref{prop:ffstar}, as, e.g.,
$\Der_{M_{2,2}}(-\log
(x_{11}x_{12}(x_{11}x_{22}-x_{12}x_{21})))\subseteq
\Der_{M_{2,2}}(-\log (x_{12},x_{22}))$.
\end{example}

\section{The case of maps
\texorpdfstring{$\vp:\CC^{n+1}\to\CC^{n}$}{phi:C{\textasciicircum}(n+1)->C{\textasciicircum}n}}
\label{sec:cnp1tocn}
We now show that for
germs $\vp:X=\CC^{n+1}\to S=\CC^{n}$ with critical set of
codimension $2$,
the $\mathcal O_{X}$--module
$T^1_{X/S}$ is Cohen--Macaulay of codimension
$2$.
In fact, this is the idea behind Theorem \ref{fibres}, concerning the versal
deformations of isolated complete intersection curve singularities.

\begin{proposition}
\label{prop:nnmo}
Let $\varphi:X=\CC^{n+1}\to S=\CC^{n}$ be
holomorphic with critical set $C(\varphi)\subseteq \CC^{n+1}$.
If $C(\varphi)$ is nonempty and has codimension 2, then $T^1_{X/S}$ is a
Cohen--Macaulay
$\calox$--module of codimension 2.
The vertical vector fields form the free $\calox$--module of rank $1$
generated by
$\eta=\sum_{i=1}^{n+1} (-1)^i d_i \frac{\partial}{\partial x_i}$,
where $d_i$ is the determinant of $\Jac(\varphi)$ with column $i$
deleted.
\end{proposition}

\begin{proof}
\cite[Proposition 6.12]{Looijenga}
uses the Buchsbaum--Rim complex to prove
that for $g:\CC^p\to\CC^r$, $p\geq r$, if $C(g)$ has
the expected
dimension $r-1$,
then $\coker(\Jac(g))$ is a Cohen--Macaulay $\calo_{\CC^p}$--module of
dimension $r-1$. 

Thus, in the case at hand,
$T^1_{X/S}\cong \coker(\Jac(\vp))$ is Cohen--Macaulay of
codimension $2$, and the
Buchsbaum--Rim complex for $\bigwedge^1 \Jac(\varphi)=\Jac(\varphi)$ is 
exact and of the form
\begin{equation}
\label{xy:cnnmo}
\xymatrix{
0\ar[r]%
& \calox \ar[r]^-{\cdot\epsilon\eta}
& \left(\calox\right)^{n+1} \ar[r]^-{\Jac(\varphi)}
& \left(\calox\right)^{n} \ar[r]
& T^1_{X/S} \ar[r] & 0
},
\end{equation}
where $\epsilon=(-1)^{\binom{n+2}{2}}$.
Hence
$T^0_{X/S}$ is the free module generated by $\eta$.
\end{proof}

\begin{example}
Let $\varphi:\CC^3\to\CC^2$ be defined by
$\varphi(x_1,x_2,x_3)=(x_1^2+x_2^3,x_2^2+x_1x_3)$.
The critical locus $V(x_1,x_2x_3)$ has codimension 2,
and the discriminant is 
the plane curve defined by $\Delta=s_1^2-s_2^3$.
A Macaulay2 \cite{M2} computation shows that the 
liftable vector fields are exactly
$\Der_S(-\log \Delta)$.
By 
Proposition \ref{prop:nnmo}
and Theorem \ref{mthm}, 
we conclude that $\varphi^{-1}(\Delta)$ is a free divisor
defined by
$$\Delta\varphi=
x_1(-3x_2^4x_3-3x_1x_2^2x_3^2-x_1^2x_3^3+2x_1x_2^3+x_1^3).$$
A generating set of $\Der_X(-\log \Delta\vp)$ consists of
lifts of a generating set of $\Der_S(-\log \Delta)$, and
the vertical vector field
$
-3x_1x_2^2\frac{\partial}{\partial x_1}
+2x_1^2\frac{\partial}{\partial x_2}
-(4x_1x_2-3x_2^2x_3)\frac{\partial}{\partial x_3}
$.
\end{example}

\begin{example}
Let $\vp:\CC^4\to \CC^3$ be defined by
$\vp(x_1,x_2,x_3,x_4)=(x_1x_3,x_2^2-x_3^3,x_2x_4)$.
The critical locus $C(\vp)$ has codimension $2$,
and so by Proposition \ref{prop:nnmo} the module
$T^1_{X/S}$ is Cohen--Macaulay of
codimension $2$. 
Although the module of all liftable vector fields is not free,
thus not associated to a free divisor,
each $s_i\frac{\partial}{\partial s_i}$, $i=1,2,3$, is liftable.
Hence, any free divisor in $\CC^3$ containing the normal crossings
divisor $s_1s_2s_3=0$ will lift via $\vp$ to a free divisor in
$\CC^4$.
\end{example}

\section{Coregular and Cofree Group Actions}
\label{sec:groupactions}

For a reductive linear algebraic group $G$ acting on $X$, we now
consider the algebraic quotient
$\vp:X\to S=X\slashslash G$.

\subsection*{Castling}
Our initial example is related to 
the classical castling of prehomogeneous vector spaces.
\begin{sit}
Let $G=\SL(n,\CC)$ act on the affine space $V=M_{n,n+1}$ of $n\times (n+1)$ matrices over $\CC$ by left
multiplication.
Use coordinates $\{x_{ij}:1\leq i\leq n, 1\leq j\leq n+1\}$ for $V$,
and let $\Delta_i$ be $(-1)^i$ times the
$n\times n$ minor obtained by deleting the $i$th column of the generic matrix 
$(x_{ij})$.
The quotient space $V\slashslash G$ is then again smooth and the corresponding invariant ring
$R=\CC[V]^{G}$ is the polynomial ring on the $n\times n$ minors 
$\{\Delta_{i}\}_{i=1,...,n+1}$ (e.g., \cite[\S9.3,9.4]{PV}).
In particular, $\dim R=n+1$, and the quotient map
$\vp:V\to V\slashslash G$ is smooth outside the null cone $\vp^{-1}(0)$ that in turn is the determinantal variety defined
by the vanishing of the maximal minors of the generic matrix, thus, Cohen--Macaulay of codimension $2$.
\end{sit}
\begin{theorem}
\label{thm:maxmin}
Let $f\in \calos$ define a free divisor in $S=\CC^{n+1}$ that is not suspended, equivalently \cite{GS}, 
$\Der_{S}(-\log f)\subseteq \fm_{S}T^{0}_{S}$. Then $f(\Delta_{1},...,\Delta_{n+1})$ defines a free divisor on 
$\CC^{n(n+1)}$.
\end{theorem}

\begin{proof}
Let $X=V\cong \CC^{n(n+1)}, S=V\slashslash G\cong \CC^{n+1}$ and
let $\vp:X\to S$ be the natural morphism,
smooth off the codimension $2$ null cone $\vp^{-1}(0)$.

That the Kodaira--Spencer map restricted to the logarithmic vector fields
along $f$ vanishes is due to our assumption that $\Der_{S}(-\log f)\subseteq \fm_{S}T^{0}_{S}$ and the fact that 
we can exhibit lifts of a generating set of $\fm_{S}T^{0}_{S}$.  Indeed, a computation 
shows that for $1\leq p,q\leq n+1$ with $p\neq q$ and any $1\leq r\leq n$,
\begin{equation}
\begin{split}
\label{eqn:castlinglifts}
\Jac(\varphi)\left(
 -\sum_{i=1}^n x_{iq} \frac{\partial}{\partial x_{ip}}
\right)
&=\Delta_p \frac{\partial}{\partial s_q}
=T^0_S(\varphi^{\flat})\left(s_p \frac{\partial}{\partial s_q}\right)
\\
\Jac(\varphi)\left(
 \sum_{j=1}^{n+1} x_{rj} \frac{\partial}{\partial x_{rj}}
 -
 \sum_{i=1}^n x_{iq} \frac{\partial}{\partial x_{iq}}
\right)
&=\Delta_q \frac{\partial}{\partial s_q}
=T^0_S(\varphi^{\flat})\left(s_q\frac{\partial}{\partial s_q}\right).
\end{split}
\end{equation}
(In each case, $\Jac(\varphi)$ applied to the sum over $i$ gives a sum
where the coefficient of $\frac{\partial}{\partial s_k}$ is of the form
$\sum_{i=1}^n x_{iq} \frac{\partial \Delta_k}{\partial x_{ip}}$,
which simplifies
to $\pm \Delta_p$, $\pm \Delta_k$, or $0$,
depending on $p,q,k$.
Applying $\Jac(\vp)$ to the sum over $j$ gives $\sum_{k=1}^{n+1}
\Delta_k \frac{\partial}{\partial s_k}$, as each minor is linear in
row $r$.
Or, see \ref{sit:explainslvfs}.) 
%
This shows that 
condition \eqref{mthm.ks} of Theorem \ref{mthm} is satisfied.

It suffices to establish condition \eqref{mthm.t1xs}. This will follow
from the dual Zariski--Jacobi sequence, once we show that the $\calox$--module $T^{0}_{X/S}$ 
of vertical vector fields along the map $\vp$ is free. However, the Lie algebra $\sln$ acts through
derivations on $\calox$, defining a $\calox$--linear map $\sln\otimes \calox\to T^{0}_{X/S}$.
This map is an isomorphism outside the null cone, as the smooth fibres there are regular orbits for the 
$\SL(n,\CC)$--action. Now both source and target of the exhibited map are reflexive $\calox$--modules
and the map is an isomorphism outside the null cone of codimension $2$, whence it must be an isomorphism 
everywhere.
\end{proof}

\begin{remark}
Two types of vector fields on $M_{n,n+1}$ generate
$\Der(-\log f\vp)$.
The first are lifts of a generating set of $\Der(-\log f)$,
which may be found using \eqref{eqn:castlinglifts}.
The second are the linear vector fields
arising from the $\SL(n,\CC)$ action on $M_{n,n+1}$; these generate
the module
$T^0_{X/S}$ of vertical vector fields.
Note that this is a minimal generating set, and that
if the generators of $\Der(-\log f)$ are linear vector fields then
$\Der(-\log f\vp)$ is also generated by linear vector fields.
\end{remark}

\begin{example}
The normal crossings divisor in $S=\CC^{n+1}$ is the linear free divisor
defined by $f=s_1\cdots s_{n+1}=0$.
By Theorem \ref{thm:maxmin}, this pulls back to the linear free divisor
$f\varphi=\Delta_1\cdots \Delta_{n+1}=0$,
previously seen in \cite[7.4]{buchweitz-mond}.
A generating set of $\Der(-\log f\varphi)$ consists of the $n^2-1$ vector
fields arising from the $\SL(n,\CC)$ action on $M_{n,n+1}$, and 
lifts (as in \eqref{eqn:castlinglifts}) of the $n+1$ generators
$\left\{s_i\frac{\partial}{\partial s_i}\right\}_{i=1}^{n+1}$ of
$\Der_S(-\log f)$.
\end{example}

%
\begin{example}
Let $f=0$ be a reduced defining equation of a free surface in $\CC^3$
which is not suspended.
Such free surfaces exist in abundance, see, for example,
\cite{Damon-equisingular,Sekiguchi}.
Pulling back $f$ via $\varphi:M_{2,3}\to M_{2,3}\slashslash
\SL(2,\CC)\cong \CC^3$ produces the
free divisor
$$
f\left(
 -(x_{12}x_{23}-x_{13}x_{22}),
 (x_{11}x_{23}-x_{13}x_{21}),
 -(x_{11}x_{22}-x_{12}x_{21}) \right)=0$$
in $M_{2,3}$.
For instance, $f=s_1(s_1s_3-s_2^2)$ pulls back to the linear free divisor
\begin{multline*}
(x_{12}x_{23}-x_{13}x_{22})\cdot(-x_{12}x_{23}x_{11}x_{22}+x_{12}^2x_{23}x_{21}+x_{13}x_{22}^2x_{11}-x_{13}x_{22}x_{12}x_{21}
\\+x_{11}^2x_{23}^2-2x_{11}x_{23}x_{13}x_{21}+x_{13}^2x_{21}^2)=0.
\end{multline*}
\end{example}

\begin{sit}
The classical castling construction relates a representation $\rho$ of a group $G$ on $M_{n,m}$, $m<n$, to a representation $\rho'$ of some
$G'$ on $M_{n,n-m}$, and vice versa.
Then $\rho$ has a Zariski open orbit if and only if $\rho'$ has a
Zariski open orbit, and the
hypersurface component of the complement of each is defined by a homogeneous
polynomial ($H$, respectively, $H'$) in the respective generic
maximal minors
(\cite[\S2.3]{granger-mond-schulze}).
There is a bijection between the maximal minors of $M_{n,m}$ and
$M_{n,n-m}$ defined by replacing a $m\times m$ minor $\Delta_I$ on
$M_{n,m}$ with the $(n-m)\times (n-m)$ minor $\Delta'_I$ on
$M_{n,n-m}$ formed by using the complementary set of rows and an
appropriate sign.
As polynomials in the minors, via this correspondence $H$ and $H'$ are
the same up to multiplication by a unit.

Castling sends linear free divisors to linear free divisors
by \cite[Prop.\ 2.10(4)]{granger-mond-schulze}.
For arbitrary free divisors, our Theorem \ref{thm:maxmin} addresses the
$n=m+1$ situation (in one direction), and it is reasonable to ask whether it holds more generally
for arbitrary $(n,m)$.
One difficulty is that there is generally no morphism
between $M_{n,m}$ and $M_{n,n-m}$
which sends $\Delta_I$ to $\Delta'_I$, or vice-versa, and hence
it is unclear how to lift vector fields, or even what this means.
In the classical situation, an underlying representation $\theta$ of a group $H$
on a $n$-dimensional space is used in the construction of both $\rho$ and $\rho'$,
and so gives a correspondence between the vector fields generated by
the action of $\theta$ on
the two spaces.
\end{sit}

The general situation remains mysterious:

\begin{example}
For $(n,m)=(5,2)$, let $\Delta_{ij}$ denote the minor on $M_{5,2}$
obtained by using only rows $i$ and $j$.
A calculation using the software Macaulay2 or Singular shows that
$\Delta_{14}\Delta_{15}(\Delta_{14}\Delta_{25}-\Delta_{15}\Delta_{24})
(\Delta_{34}\Delta_{45}-\Delta_{35}^2)=0$
defines a (non-linear) free divisor on $M_{5,2}$.
Another computation shows that the corresponding divisor on $M_{5,3}$ is not free.
It is unclear what additional hypotheses are necessary to generalize
Theorem \ref{thm:maxmin}.
%
\end{example}

\subsection*{Group actions}
We now generalize the ideas behind Theorem \ref{thm:maxmin} to
the case when $\vp:X\to S$ is given by the quotient of $X$ under a group action.
We work now in the algebraic category of schemes of finite type over
$\CC$.
Recall the following definitions.

\begin{defn}
If $G$ is any reductive complex algebraic group, then a finite dimensional linear representation $V$ is 
\begin{enumerate}[\rm(a)]
\item {\em coregular\/} if the quotient space $V\slashslash G$ is {\em smooth\/};
\item {\em cofree}, if further the natural 
projection $\vp:V\to V\slashslash G$ is {\em flat}, equivalently (see
\cite[\S 8.1]{PV}),  
$\vp:V\to V\slashslash G$ is coregular and {\em equidimensional\/} in that all fibres have the same dimension;
\item {\em coreduced\/}, if the {\em null cone\/} $\vp^{-1}(0)$ is reduced.
\end{enumerate}
\end{defn}

In algebraic terms, with $\CC[V]$ the ring of polynomial functions, coregularity means that 
the ring of invariants $R=\CC[V]^{G}$ is again a polynomial ring, while cofreeness means that further 
$\CC[V]$ is free as an $R$--module (e.g., \cite[\S8.1]{PV}\footnote{%
The reference there for the algebraic result needed to justify this
interpretation of cofreeness is incorrect,
and should be Bourbaki's \emph{Groupes et Alg\`ebres de Lie}, Chap.~V,
\S5, Lemma 1.
}).
If $R=\CC[f_{1},..., f_{d}]$ is the polynomial ring 
over the indicated invariant functions $f_{j}\in \CC[V]$, then in the cofree case these functions form a 
regular sequence in $\CC[V]$.

\begin{remark}
A famous conjecture by Popov suggests that equidimensionality of (the fibres of) the projection 
$\vp:V\to V\slashslash G$ already implies coregularity and then
automatically cofreeness for $G$ connected semi-simple.
\end{remark}

There are many examples of cofree representations, and even more that are coregular.
We just mention Kempf's basic result that a representation is automatically cofree whenever 
$\dim V\slashslash G \leqslant 2$; see \cite[Thm.\ 8.6]{PV} or \cite{Kempf}. For further lists of such representations 
see \cite{Schwarz,Littelmann,Wehlau}.

\begin{remark}
In the case of Theorem \ref{thm:maxmin} above, the action of
$\SL(n,\CC)$ on $M_{n,n+1}$ is 
coregular, but not cofree.
\end{remark}

\begin{sit}
To apply our main theorems to the quotient
$X\to S$
of a coregular representation, $T^0_{X/S}$
must be free.
There is a straightforward sufficient criterion for
the stronger condition
that $T^1_{X/S}$ is Cohen--Macaulay of codimension $2$.
\end{sit}

\begin{proposition}
\label{prop:coregulart1cm}
Let $X=V$ be a coregular representation of the
reductive complex algebraic
group $G$ with Lie algebra $\fg$ and quotient $S=V\slashslash G$.
If the generic stabilizer of $G$ on $X$ is of dimension $0$ and
the natural morphism
$\vp:X\to S$ is smooth outside a set of codimension $2$ in $X$, then
\begin{enumerate}[\rm(i)]
\item
\label{conc:altgoxiso}
The natural
$\calox$--homomorphism $\fg\otimes\calox\to T^{0}_{X/S}$ is an
isomorphism; and
\item
\label{conc:altiscmcodim2}
$T^1_{X/S}$ is a Cohen--Macaulay $\calox$--module of codimension $2$.
\end{enumerate}
\end{proposition}
\begin{proof}
$\vp$ is smooth outside of a set of codimension $2$ in $X$
and $T^1_{X/S}$ is supported on the critical locus of $\vp$,
so $\codim(\supp T^1_{X/S})\geq 2$,
or $\dim(T^1_{X/S})\leq \dim(X)-2$.

Since the generic stabilizer of 
$G$ on $X$ is of dimension zero, thus, a finite group,
the $\calox$--homomorphism 
$\rho:\fg\otimes \calox\to T^{0}_{X/S}$ is an inclusion.
On the set in $X$ where $\varphi$ is smooth,
$\rho$ is also locally surjective.
Since $T^0_{X/S}$ is a second syzygy module by the
dual Zariski--Jacobi sequence,
it is reflexive (e.g., \cite[Prop.\ 1.1]{hartshorne-stablereflexive}).
%
As $\fg\otimes \calox$ is free,
$\rho$ is a homomorphism between 
reflexive modules which is an isomorphism off a set of
codimension $\geq 2$, and hence $\rho$ is an isomorphism.
This proves \eqref{conc:altgoxiso}.

By \eqref{conc:altgoxiso} and the dual Zariski--Jacobi sequence,
$\projdim_{\calox} T^1_{X/S}\leq 2$.
By the Auslander--Buchsbaum formula and the usual
relation between depth and dimension,
$$
\dim(X)-2\leq \depth(T^1_{X/S})\leq \dim(T^1_{X/S}).$$
As 
already $\dim(T^1_{X/S})\leq \dim(X)-2$,
$T^1_{X/S}$ is Cohen--Macaulay of codimension $2$.
\end{proof}

There are coregular representations for
which $T^0_{X/S}$ of the quotient is free, but
$T^1_{X/S}$ is not Cohen--Macaulay of codimension $2$
(e.g., Example \ref{ex:square}).
Our next result gives some insight into these cases, and also
gives a necessary numerical condition for Proposition
\ref{prop:coregulart1cm} to apply;
we have used it to choose our examples.

\begin{proposition}
\label{prop:numericalcriterion}
Let $X=V$ be a coregular representation of the
reductive complex algebraic
group $G$ with Lie algebra $\fg$ and quotient $S=V\slashslash G$.
Let $N=\dim(X)$, $d=\dim(S)$, and let
$\delta_1,\ldots,\delta_d\geqslant 1$
be the degrees of the generating invariants.
If the natural 
$\calox$--homomorphism $\fg\otimes\calox\to T^{0}_{X/S}$ is an
isomorphism,
then either
\begin{enumerate}
\item
$N=\sum_{\nu=1}^d \delta_\nu$ and $\dim(T^1_{X/S})=N-2$; or
\item
$N\neq \sum_{\nu=1}^d \delta_\nu$ and $\dim(T^1_{X/S})=N-1$.
\end{enumerate}
\end{proposition}
\begin{proof}
If $T^0_{X/S}$ is free and generated by the group action, then
the dual Zariski--Jacobi sequence provides
a graded free resolution of the graded $\calox$--module $T^{1}_{X/S}$
of the form
\begin{equation}
\label{eqn:grres}
\xymatrix{
0
\ar[r] &
\calo_{X}^{\oplus (N-d)}
\ar[r] &
\oplus_{\nu=1}^{N}\calox(1)
\ar[r] &
\oplus_{\nu=1}^{d}\calox(\delta_{\nu})
\ar[r] &
T^{1}_{X/S}
\ar[r] &
0\,.
}
\end{equation}
First, \eqref{eqn:grres} implies 
$\projdim_{\calox}(T^1_{X/S})\leq 2$,
and then the Auslander--Buchsbaum formula
shows
$\dim(T^1_{X/S})\geq \dim(X)-2$.
Also by \eqref{eqn:grres},
and using the identity
$t^{-n}-1=(1-t)(t^{-1}+\cdots+t^{-n})$,
the Hilbert--Poincar\'e series of $T^{1}_{X/S}$ satisfies
\begin{align*}
\HH_{T^{1}_{X/S}}(t)
&= \frac{1}{(1-t)^{N}}\left( \sum_{\nu=1}^{d}t^{-\delta_{\nu}}-N t^{-1}+N-d\right)\\
&= \frac{1}{(1-t)^{N}}\left( \sum_{\nu=1}^{d}(t^{-\delta_{\nu}}-1)-N(t^{-1}-1)\right)\\
&= \frac{\left(\left(\sum_{\nu=1}^{d}\delta_{\nu}\right)-N\right) +(1-t)p(t,t^{-1})}{(1-t)^{N-1}}
\end{align*}
for some Laurent polynomial $p(t,t^{-1})\in\ZZ[t,t^{-1}]$.
In particular, $\HH_{T^1_{X/S}}$ has a pole at $t=1$ of order $N-1$ 
if and only if
$N\neq \sum_{\nu=1}^d \delta_\nu$.
Finally, the order of this pole equals
$\dim(T^1_{X/S})$.
\end{proof}

\begin{remark}
If we inspect the table
of cofree irreducible representations of simple groups
in \cite[Summary Table]{PV}, we check readily that
when the generic stabilizer is finite,
the equation $\dim(X)=\sum_{\nu=1}^d \delta_\nu$
is satisfied.
However, the tables 
of cofree irreducible representations of semisimple groups in
\cite{Littelmann}
show that this is not automatic; 
for instance (in the notation there),
the representation $\omega_5+\omega'_1$ of
$B_5+A_1$ has $\dim(X)=64$ and $(\delta_i)=(2,4,6,8,8,12)$, which
falls $64-40=24$ short.
%
\end{remark}

\begin{sit}
To prove that vector fields lift across $\vp:X\to X\slashslash G$,
the following technique is useful.
\end{sit}

\begin{lemma}
\label{lemma:grouplift}
Let $X=V$ be a coregular representation of the algebraic group $G$
with quotient $S=V\slashslash G$.  Let $\rho_X$ and $\rho_S$ be representations
of an algebraic group $H$ on $X$, respectively, $S$.
If $\vp:X\to S$ is equivariant with respect to the action of $H$,
then all vector fields on $S$
obtained by differentiating $\rho_S$ lift across $\vp$.
\end{lemma}
\begin{proof}
Differentiating gives representations $d\rho_X$ and $d\rho_S$ of
$\fh$, the Lie algebra of $H$, as Lie algebras of vector fields on
$X$, respectively, $S$.
Since $\vp$ is equivariant, for each $Y\in\fh$, $d\rho_X(Y)$ is
$\vp$-related to $d\rho_S(Y)$, and hence $d\rho_S(Y)$ lifts to
$d\rho_X(Y)$.
\end{proof}

\begin{example}
\label{sit:explainslvfs}
This argument may be used in
the castling situation of Theorem \ref{thm:maxmin}.
There, $\GL({n+1},\CC)$ has representations $\rho_X$ and $\rho_S$ on
$X=M_{n,n+1}$ and $S=M_{1,n+1}$ defined by
$$\rho_X(A)(B)=BA^T
\text{ and }
\rho_S(A)(C)=C\adj(A)=C\det(A)A^{-1},
$$
where $\adj(A)$ is the adjugate of $A$.
(If $M_{n,n+1}\simeq V\otimes W$, with $\dim(V)=n$, $\dim(W)=n+1$,
and $M_{1,n+1}\simeq \CC\otimes W^*$, then
$\rho_X$ is a representation on $W$ and $\rho_S$ is the contragredient
representation of $\rho_X$.)
A calculation shows that
$\vp$ is equivariant with respect to $\rho_X$ and $\rho_S$.
Since $d\rho_S$ produces 
a generating set of $\Der_S(-\log \{0\})$,
any $\eta\in\Der_S(-\log \{0\})$ will lift.
(Note that the lifts in \eqref{eqn:castlinglifts} have been
simplified.)
%
%
\end{example}

\begin{sit}
We now investigate a method for determining the generating invariants
of lowest degree.
First we observe that if $\Jac(\vp)$ was known, then
it would be easy to determine a generating set of invariants.
\end{sit}


\begin{proposition}
Let $X=V$ be a coregular representation of $G$, and let
$\vp:X\to S=V\slashslash G$ with $N=\dim(X)$ and $d=\dim(S)$.
If $E=\sum_{i=1}^N x_i \frac{\partial}{\partial x_i}\in T^0_X$ is the Euler
vector field and we write
$$\Jac(\vp)(E)=\sum_{j=1}^d \tilde{f}_j \frac{\partial}{\partial s_j}\in T^0_S(\calox),$$
then the coefficient functions $\tilde{f}_{j}$ form a generating set of the invariants in 
$\CC[V]^{G}\subseteq \CC[V]$.
\end{proposition}

\begin{proof}
Observe that $\vp^{\flat}:\calos=\CC[s_{1},\ldots,s_{d}]\to \CC[f_{1},\ldots,f_{d}]=\CC[V]^{G}\subseteq \CC[V]$ is the 
canonical inclusion, that is, $\vp^{\flat}(s_{j})=f_{j}$.  Since each $f_j$ is homogeneous, we have
\begin{align}
\label{eqn:aneulervflift}
\Jac(\vp)(E)
&=
\sum_{j=1}^d \left(\sum_{i=1}^N x_i \frac{\partial (s_j\circ \vp)}{\partial x_i} \right)\frac{\partial}{\partial s_j}
=
\sum_{j=1}^d \deg(f_j) f_j\frac{\partial}{\partial s_j}.
\end{align}
Now each $\deg(f_{j})>0$, and we are in characteristic zero, so that the functions 
$\tilde{f}_{j}= \deg(f_{j})f_{j}$ also form a generating set of invariants.
\end{proof}

\begin{remark}
In the proof, \eqref{eqn:aneulervflift} shows that
$\sum_{j=1}^d \deg(f_j)s_j \frac{\partial}{\partial s_j}\in T^0_S$
lifts across $\vp$, although we do not use this fact.
\end{remark}

\begin{sit}
We now describe a way to compute the $\calox$--ideal generated by the
invariants, even without knowledge of the invariants.
 From this ideal we may recover the 
invariants of lowest degree. 
\end{sit}

\begin{proposition}
Let $X=V$ be a coregular representation of $G$ with finite generic stabilizer,
and let $\vp:X\to S=X\slashslash G$, with $N=\dim(X)$.
Let $f_1,\ldots,f_d$ be generating invariants, and
let $J=(f_1,\ldots,f_d)\calox$.
Let $K\subseteq (\calox)^N$ be the $\calox$--module 
of $b=(b_i)$ such that $\sum_{i=1}^N b_i a_i=0$ for any linear vector field
$\sum_{i=1}^N a_i \frac{\partial}{\partial x_i}$ on $X$ arising from the
action of the Lie algebra $\fg$ of $G$.
Let $I$ be the $\calox$--ideal consisting of $\sum_{i=1}^N b_i
x_i$, where $(b_i)\in K$.
If $T^1_{X/S}$ is a Cohen--Macaulay $\calox$--module of codimension
$2$, then $J=I$.
\end{proposition}

\begin{proof}
For a homogeneous invariant $g\in\calox$, $(\frac{\partial g}{\partial x_i})\in K$,
and hence $g\in I$.  It follows that $J\subseteq I$.

Since $\dim(T^1_{X/S})$ is the dimension of the critical locus,
$\vp$ is smooth off a set of codimension $2$.
By Proposition \ref{prop:coregulart1cm}, $T^0_{X/S}$ is generated by
the action $\theta:\fg\to\gl(V)\cong V\otimes V^{\vee}$ of $\fg$.
As $\calox=\CC[V]$,
the first map $\rho$ in the dual Zariski--Jacobi sequence
\begin{align}
\label{al:prop415les}
\xymatrix{
0\ar[r]&\fg\otimes \calox \ar[r]^-{\rho}& T^0_X \ar[r]^-{\Jac(\vp)}&
T^0_S(\calox)\ar[r]&
T^{1}_{X/S}\ar[r]&0
}
\end{align}
is given by the composition
$$
\xymatrix{
\fg\otimes \CC[V]
\ar[r]^-{\theta\otimes 1} &
V\otimes V^\vee\otimes \CC[V]
\ar[r] & 
V\otimes \CC[V](1)
\cong T^0_X.
}
$$
Split \eqref{al:prop415les} into short exact sequences and
take $\calox$--duals
to get the exact sequence
\begin{align}
\label{xy:ses}
\xymatrix{
0\ar[r]
& N^* \ar[r]^{\psi}
& \Omega^1_X \ar[r]^-{\rho^*}
& \fg^*\otimes \calox,
}
\end{align}
where
$N=\mathrm{Image}(\Jac(\vp))$,
$\psi$ is the dual of $\Jac(\vp)$,
and $\Omega^1_X\cong (T^0_X)^*=(\Omega^1_X)^{**}$
as the smoothness of $X$ implies the reflexivity of $\Omega^1_X$.
By this identification, 
the Euler derivation $E\in T^0_X$ gives a map $\tilde{E}:\Omega^1_X\to \calox$
defined by $\tilde{E}\left(\sum a_i dx_i\right)=\sum a_i x_i$.

Observe that
under the obvious identification of $(\calox)^N$ with $\Omega^1_X$,
$K\cong \ker(\rho^*)$, and $I=\tilde{E}(\ker(\rho^*))$.
Let $b\in \ker(\rho^*)$.
By the exactness of \eqref{xy:ses}, there exists an
$n\in N^*$ such that $b=\psi(n)$.
Then by the form of $\psi$ and the homogeneity of $f_1,\ldots,f_d$, we have
$$
\begin{pmatrix} \tilde{E}(b) \end{pmatrix}
=
\begin{pmatrix}
x_1 & \cdots & x_n
\end{pmatrix}
\begin{pmatrix}
b_1 \\ \vdots \\ b_n
\end{pmatrix}
=
\begin{pmatrix}
x_1 & \cdots & x_n
\end{pmatrix}
\begin{pmatrix}
\frac{\partial f_1}{\partial x_1}
& \cdots &
\frac{\partial f_d}{\partial x_1} \\
\vdots & \ddots & \vdots \\
\frac{\partial f_1}{\partial x_n}
& \cdots & 
\frac{\partial f_d}{\partial x_n}
\end{pmatrix}
\begin{pmatrix}
n_1 \\ \vdots \\ n_d
\end{pmatrix}
\in J.\qedhere$$
\end{proof}

\section{Examples of group actions}
\label{sec:groupexamples}
\begin{sit}
We now apply the results of \S\ref{sec:groupactions} to a number of 
coregular and cofree group actions.
For each $\vp:X\to S=X\slashslash G$, we determine when
$T^1_{X/S}$ is Cohen--Macaulay of codimension $2$, and determine the
liftable vector fields.
These examples come from classifications
that provide the number and degrees of
the generating invariants.

To check our hypotheses for $\vp$, however, it is necessary to choose specific
generating invariants.
For many of the examples below, we have used Macaulay2 \cite{M2} to
find all invariants of the given degrees%
\footnote{The vector space of degree $d$ invariants of a
linear representation of a connected group is just the
space of degree $d$ polynomials annihilated
by the linear vector fields corresponding to the Lie algebra action.},
make a choice of generating invariants to find an explicit form for
$\vp$, 
compute the dimension of the critical locus of $\vp$, and find the module
of liftable vector fields.
A different choice of generating invariants gives a 
different presentation of $\CC[X]^G$ as a polynomial ring, a new
$\vp'$ (equal to $\vp$ composed with a diffeomorphism in $S$),
and a different module of liftable vector fields.

Note also that there are many other examples in, e.g., \cite{Littelmann}.
\end{sit}

\subsection*{Special linear group}

\begin{example}
Let $\rho:\SL(2,\CC)\to\GL(V)$, $V=\CC x\oplus \CC y$, be 
the standard representation of $G=\SL(2,\CC)$.
Differentiating this representation gives the vector fields
\begin{equation}
\label{eqn:sl2diffops}
d\rho(e)=x\partial_y,\quad d\rho(f)=y\partial_x,\quad d\rho(h)=x\partial_x-y\partial_y
\end{equation}
on $V$, where $\sltwo=\CC\{e,f,h\}$.

Consider the 
$n$th symmetric power $X=\Sym^n(V)$ of $\rho$,
where $\Sym^n(V)$ has 
the $\CC$-basis $z_i=x^{n-i}y^i$ for $i=0,\ldots,n$.
Differentiating this $G$--representation shows that
$e$, $f$, and $h$ act on each $x^{n-i}y^i$ by the corresponding
differential operator in \eqref{eqn:sl2diffops}.
Let $\vp:X\to X\slashslash G=S$.

For $1\leq n\leq 4$, the resulting representation appears in the list of
cofree representations of \cite{Littelmann},
along with the dimension $g$ of the generic isotropy subgroup,
and the number ($=\dim(S)$) and degrees of the generating invariants.
For $n=1,2$, Proposition \ref{prop:coregulart1cm} does not
apply because $g=1$.

When $n=3$, then $X$ is the space of so-called  ``binary cubics'',
and
$g=0$, $\dim(S)=1$. As
a sole generating invariant one can take
$$f_1=-3z_1^2z_2^2       +4z_0z_2^3     +4z_1^3z_3     -6z_0 z_1 z_2 z_3+z_0^2z_3^2.$$
Since 
it is readily checked that $\varphi=(f_1)$ is smooth off a set of
codimension $2$, it follows from Proposition \ref{prop:coregulart1cm}
that $T^1_{X/S}$ is Cohen--Macaulay of
codimension $2$.  
We compute that 
any $\eta\in\Der_S(-\log s_1)$ will lift, 
so Theorem \ref{mthm} implies that $f_1$ itself
will define a free divisor.
Note that a linear change of coordinates takes $f_1$ to
the example 
\cite[Ex.\ 1.4(2) or \S6.4]{gmns}.

When $n=4$, the case of ``binary quartics'', then
$g=0$, $\dim(S)=2$, and
the generating invariants are
$$
f_1=3z_2^2-4z_1z_3+z_0z_4
\qquad
\mathrm{and}
\qquad
f_2=z_2^3-2z_1z_2z_3+z_0z_3^2+z_1^2z_4-z_0z_2z_4.
$$
The map $\varphi=(f_1,f_2)$ is smooth off a set of codimension $2$, so by 
Proposition \ref{prop:coregulart1cm}, the module $T^1_{X/S}$ is Cohen--Macaulay
of codimension $2$.
(As expected by Proposition \ref{prop:numericalcriterion},
$\dim(X)=5=\deg(f_1)+\deg(f_2)$.) 
The liftable vector fields are $\Der_S(-\log
(s_1^3-27s_2^2))$.
Since all reduced plane curve singularities are free divisors,
by Theorem \ref{mthm} any reduced plane curve containing 
$s_1^3-27s_2^2$ as a component lifts through this group action to a free divisor in 
$\Sym^4(V)\cong \CC^5$.
\end{example}

\begin{example}
Let $V$ be the standard representation of $G=\SL(3,\CC)$.
Then $X=\Sym^3(V)\cong \CC^{10}$, the space of ``ternary cubics'', has finite generic isotropy subgroup,
$S=X\slashslash G$ has dimension $2$, and the invariants $g_S$, $g_T$ have degree 4 and 6
(e.g., \cite[4.4.7, 4.5.3]{sturmfels}).
Then $\varphi=(g_S,g_T)$ is smooth off a set of
codimension $2$, and the liftable vector fields are exactly
$\Der_S(-\log (64s_1^3-s_2^2))$.
By Proposition \ref{prop:coregulart1cm} and
Theorem \ref{mthm}, any reduced plane curve singularity
which contains $64s_1^3-s_2^2$ as a component lifts via $\varphi$ to a free divisor in $X$.
\end{example}

\begin{example}
Let $V$ be the standard representation of $\SL(2,\CC)$, and
let $X=\Sym^2(V)\otimes \Sym^2(V)$, a representation of
$G=\SL(2,\CC)\times \SL(2,\CC)$.
Use the basis $y_{ij}=x_1^ix_2^{2-i}\otimes x_1^j x_2^{2-j}$,
$0\leq i,j\leq 2$, for $X$.
By \cite{Littelmann}, this cofree representation
has finite generic isotropy subgroup,
$S=X\slashslash G$ of dimension $3$, with invariants
$g_2$, $g_3$, and $g_4$, $\deg(g_i)=i$.
We compute generating invariants as
\begin{align*}
g_2=&
2y_{11}^2-2y_{12}y_{10}+y_{20}y_{02}-2y_{21}y_{01}+y_{22}y_{00},
\\
g_3=&y_{20}y_{11}y_{02}-y_{21}y_{10}y_{02}-y_{20}y_{12}y_{01}+y_{22}y_{10}y_{01}+y_{21}y_{12}y_{00}-y_{22}y_{11}y_{00},
\\
g_4=&-4y_{11}^4+8y_{12}y_{11}^2y_{10}-4y_{12}^2y_{10}^2+2y_{20}y_{12}
y_{10}y_{02}-4y_{21}y_{11}y_{10}y_{02}+2y_{22}y_{10}^2y_{02}\\
&-\tfrac{1}{2}
y_{20}^2y_{02}^2-4y_{20}y_{12}y_{11}y_{01}+8y_{21}y_{11}^2y_{01}-4
y_{22}y_{11}y_{10}y_{01}+2y_{21}y_{20}y_{02}y_{01}\\
&-4y_{21}^2
y_{01}^2 +2y_{22}y_{20}y_{01}^2+2y_{20}y_{12}^2y_{00}-4y_{21}y_{12}
y_{11}y_{00}+2y_{22}y_{12}y_{10}y_{00}\\
&+2y_{21}^2y_{02}y_{00}-3y_{22}
y_{20}y_{02}y_{00}+2y_{22}y_{21}y_{01}y_{00}-\tfrac{1}{2}y_{22}^2y_{00}^2.
\end{align*}
For $\vp=(g_2,g_3,g_4)$,
$\vp$ is smooth off a set of codimension $2$ and
the liftable vector fields are
$\Der(-\log \Delta)$ for
the free divisor defined by
$\Delta=s_1^6-10s_1^3s_2^2+4s_1^4s_3+27s_2^4-18s_1s_2^2s_3+5s_1^2s_3^2+2s_3^3.$
By
Proposition \ref{prop:coregulart1cm} and Theorem
\ref{mthm}, any
free divisor in $\CC^3$ containing $\Delta$ as a component lifts to a free divisor in $X\cong \CC^9$.
Note that $\Delta$ is equivalent to the classical swallowtail.
\end{example}

\begin{example}
Let $V$ be the standard representation of $\SL(2,\CC)$, and
let $X=\Sym^3(V)\otimes V$, a representation of 
$G=\SL(2,\CC)\times \SL(2,\CC)$.
On $X$, use the basis
$z_{ij}=x_1^ix_2^{3-i}\otimes x_j$, 
where $0\leq i\leq 3$, $1\leq j\leq 2$.
By \cite{Littelmann}, this cofree representation has
finite generic isotropy subgroup, $S=X\slashslash G$ of dimension $2$, with
invariants $g_2$ and $g_6$, $\deg(g_i)=i$.
We compute the invariants as
\begin{align*}
g_2=&
3z_{22}z_{11}-3z_{21}z_{12}-z_{32}z_{01}+z_{31}z_{02},
\\
g_6=&
27z_{22}^3z_{11}^3-81z_{21}z_{22}^2z_{11}^2z_{12}+81z_{21}^2z_{22}
z_{11}z_{12}^2-27z_{21}^3z_{12}^3-27z_{32}z_{22}^2z_{11}^2z_{01}
\\&
+27z_{32}
z_{21}z_{22}z_{11}z_{12}z_{01}+27z_{31}z_{22}^2z_{11}z_{12}z_{01}+9
z_{32}^2z_{11}^2z_{12}z_{01}
\\&
-27z_{31}z_{21}z_{22}z_{12}^2z_{01}-18z_{31}
z_{32}z_{11}z_{12}^2z_{01}+9z_{31}^2z_{12}^3z_{01}+9z_{32}z_{21}z_{22}^2
z_{01}^2
\\&
-9z_{31}z_{22}^3z_{01}^2+6z_{32}^2z_{22}z_{11}z_{01}^2-15z_{32}^2z_{21}
z_{12}z_{01}^2+9z_{31}z_{32}z_{22}z_{12}z_{01}^2
\\&
-\tfrac{2}{3}z_{32}^3z_{01}^3+27z_{32}
z_{21}z_{22}z_{11}^2z_{02}-9z_{32}^2z_{11}^3z_{02}-27z_{32}z_{21}^2z_{11}
z_{12}z_{02}
\\&
-27z_{31}z_{21}z_{22}z_{11}z_{12}z_{02}+18z_{31}z_{32}
z_{11}^2z_{12}z_{02}+27z_{31}z_{21}^2z_{12}^2z_{02}-9z_{31}^2z_{11}z_{12}^2z_{02}
\\&
-18z_{32}z_{21}^2z_{22}z_{01}z_{02}+18z_{31}z_{21}z_{22}^2z_{01}
z_{02}+9z_{32}^2z_{21}z_{11}z_{01}z_{02}
\\&
-21z_{31}z_{32}z_{22}z_{11}z_{01}
z_{02}+21z_{31}z_{32}z_{21}z_{12}z_{01}z_{02}-9z_{31}^2z_{22}z_{12}z_{01}
z_{02}
\\&
+2z_{31}z_{32}^2z_{01}^2z_{02}+9z_{32}z_{21}^3z_{02}^2-9z_{31}z_{21}^2z_{22}
z_{02}^2-9z_{31}z_{32}z_{21}z_{11}z_{02}^2
\\&
+15z_{31}^2z_{22}z_{11}z_{02}^2-6
z_{31}^2z_{21}z_{12}z_{02}^2-2z_{31}^2z_{32}z_{01}z_{02}^2+\tfrac{2}{3}z_{31}^3z_{02}^3.
\end{align*}
Then $\vp=(g_2,g_6)$ is smooth off a set of
codimension $2$, and the liftable vector fields are
$\Der(-\log \Delta)$, for the plane curve
$\Delta=(s_1^3-s_2)(2s_1^3-3s_2)$.
By Proposition \ref{prop:coregulart1cm} and Theorem \ref{mthm},
any reduced plane curve containing $\Delta$ among its components
lifts to a free divisor in $X\cong \CC^8$.
In particular, $(g_2^3-g_6)(2g_2^3-3g_6)$ defines a free divisor.
\end{example}

\subsection*{Special orthogonal group}

\begin{example}
Let $V$ be the standard representation of $G=\SO(n,\CC)$.
Consider the representation $\Sym^2(V)$, which we identify with the 
action of $G$ on the space $X=\Symm_n(\CC)$ of $n\times n$ symmetric matrices by
$A\cdot M=AMA^T$.
Since multiples of the identity are fixed by $G$,
$X$ decomposes as the direct sum of the trivial $1$-dimensional representation
(on $\CC\cdot I$ for the identity $I$) and a representation (on the traceless matrices)
which appears on the lists of 
\cite[Summary Table]{PV} and \cite{Littelmann} of irreducible representations.
As a result, we know that the generic stabilizer is finite,
the generating invariants are $g_1,\ldots,g_n$, 
with $\deg(g_i)=i$,
and $S=X\slashslash G$ has dimension $n$.

Since $G$ acts by conjugation,
it preserves
the characteristic polynomial $\det(t\cdot I-M)=t^n+h_1
t^{n-1}+\cdots+h_n$ of $M$.
When restricted to the subspace $D$ of diagonal matrices,
$h_i=(-1)^i\sigma_i$, where $\sigma_i$ is the $i$th elementary
symmetric polynomial in the diagonal entries;
it follows that each $h_{k+1}\notin \CC[h_1,\ldots,h_k]$, and hence
$g_i=h_i$ are generating invariants for $i=1,\ldots,n$.
Let $\vp=(g_n,\ldots,g_1)$; under the identification of $(s_1,\ldots,s_n)\in S$
with the monic degree $n$ polynomial $t^n+s_n t^{n-1}+\cdots+s_1\in \CC[t]$, 
$\vp(M)=\det(t\cdot I-M)$.

The derivative of $\vp$ may be
computed at points having a symmetric ``Jordan form'' described in
\cite[I, \S2--3]{gantmacher};
as each $A\in\Symm_n(\CC)$ is in a $G$-orbit of such a normal form and
$\vp$ is invariant under the $G$ action,
this calculation shows that the critical locus of $\vp$ is the
($\codim \geq 2$) set of symmetric
matrices for which the Jordan canonical form has at least two Jordan blocks
with the same eigenvalue.
Then by Proposition \ref{prop:coregulart1cm}, $T^1_{X/S}$ is
Cohen--Macaulay of codimension $2$.
Also, the discriminant is the locus of monic polynomials of degree
$n$ having a repeated root, i.e., the free divisor defined by the
(classical) discriminant $\Delta$ of the polynomial
$t^n+\sum_{k=0}^{n-1} s_{k+1} t^{k}$.

Observe that $\theta=\vp|_D=(\theta_1,\ldots,\theta_n)$ is a finite map with the same discriminant,
which may be understood as the quotient $\CC^n\to \CC^n\slashslash S_n$ under the
action of the symmetric group.
Using \cite[\S1.7]{zakalyukin}, generators
for $\Der_S(-\log \Delta)$ are of the form
$\eta_k=\sum_{\ell} \alpha_{k\ell} \frac{\partial}{\partial s_\ell}$,
where $\alpha_{k\ell}\circ \theta=(\nabla \theta_k,\nabla \theta_\ell)$,
$\nabla$ is the gradient, and $(\cdot,\cdot)$ is the dot product.
Each $\eta_k$ lifts across $\vp=(\vp_1,\ldots,\vp_n)$ to
what is almost $\nabla \vp_k$,
except with off-diagonal coefficients scaled by
$\frac{1}{2}$ when the coordinates $\{x_{ij}\}_{1\leq i\leq j\leq n}$ on $\Symm_n(\CC)$
are obtained by restricting the usual coordinates on $M_{n,n}$.
By Proposition \ref{prop:coregulart1cm} and Theorem \ref{mthm},
using $\vp$ to pull back any 
free divisor containing $\Delta$ as a component produces another free
divisor.

For instance, when $n=2$ we have
$$g_1=-x_{11}-x_{22},
\qquad
g_2=x_{11}x_{22}-x_{12}^2,
\qquad\textrm{and }
\Delta=s_2^2-4s_1.$$
For $n=3$,
$$\Delta=
s_2^2s_3^2-4s_1s_3^3-4s_2^3+18s_1s_2s_3-27s_1^2.$$
For $n=4$,
\begin{align*}
\Delta=&
s_2^2s_3^2s_4^2-4s_1s_3^3s_4^2-4s_2^3s_4^3+18s_1s_2s_3s_4^3-27s_1^2s_4^4-4s_2^2s_3^3+16s_1s_3^4
\\&
+18s_2^3s_3s_4
-80s_1s_2s_3^2s_4-6s_1s_2^2s_4^2+144s_1^2s_3s_4^2-27s_2^4+144s_1s_2^2s_3
\\&
-128s_1^2s_3^2-192s_1^2s_2s_4+256s_1^3.
\end{align*}

Now let $T\subseteq \Symm_n(\CC)$ consist of the traceless matrices.
Under $\vp$, $T$ maps to the subspace $Z\subset \CC[t]$ consisting of monic
polynomials of degree $n$ with the coefficient of $t^{n-1}$ equal to
zero.
Let $\vp':T\to Z$ be defined by restricting $\vp$.
By our calculations of $d\vp$,
it follows fairly easily that $\vp'$ is a submersion at $A\in T$ if
and only if $\vp$ is a submersion at $A$.
Accordingly, $C(\vp')=C(\vp)\cap T$, and
the discriminant of $\vp'$ is
defined by $\Delta'=\Delta |_Z$.
By the same argument used for $\vp$, the map $\vp'$ is smooth off a set of
codimension $\geq 2$.
The module of logarithmic vector fields
$\Der_S(-\log \Delta)$ contains an element of the form
$\cdots+n\frac{\partial}{\partial s_n}$, so it is easy to find
$\alpha_1,\ldots,\alpha_{n-1}\in \Der_S(-\log \Delta)$
which are also tangent to $Z$, and hence restrict to elements of
$\Der_Z(-\log \Delta')$.
Note that $\vp'|_{D\cap T}$ is a finite map with the same
discriminant; by \cite{arnold},
the restrictions of
$\alpha_1,\ldots,\alpha_{n-1}$ are a free basis for $\Der_Z(-\log
\Delta')$.
Since each $\alpha_i$ lifts via $\vp$ to a vector field which is
tangent to $T$, the restriction of each $\alpha_i$ lifts via $\vp'$ to a
vector field on $T$.
Thus, by Proposition \ref{prop:coregulart1cm} and Theorem
\ref{mthm}, using $\vp'$ to pull back any free divisor
having $\Delta'$ as a component produces another free divisor.
\end{example}

\subsection*{Orthogonal group}

\begin{sit}
Let the orthogonal group
$G=\Orthog(n,\CC)$ act on the space $V=M_{n,m}$ of complex $n\times m$ matrices by multiplication on the left.
By \cite[\S9.3]{PV}, the ring of invariants is generated by 
the $\binom{m+1}{2}$ inner
products of pairs of
columns, allowing repetition.
If $\Symm_m(\CC)$ denotes the space of $m\times m$ complex symmetric
matrices, these are equivalently the 
degree 2 polynomials given by the entries of
$\varphi:X=M_{n,m}\to S=\Symm_m(\CC)$ defined by
$\varphi(A)=A^T\cdot A$.
By \cite[\S9.4]{PV}, there are relations between these generators
precisely when $n<m$.

We thus restrict ourselves to $n\geq m$,
so that $V\slashslash G\cong \Symm_m(\CC)$ is coregular and the
quotient is given by $\vp$. 
We prove some basic properties. 
\end{sit}

\begin{lemma}
\label{lemma:onproperties}
Let $\varphi$ be as above, with $n\geq m$.
Then
\begin{enumerate}[\rm(i)]
\item
\label{cond:submersion}
For $A\in X$, $\varphi$ is a submersion at $A$ if and only if
$\rank(A)=m$. 
\item
\label{cond:crit}
The critical locus $C(\varphi)=\{A\in X:\rank(A)<m\}$ has codimension $n-m+1$.
The discriminant $\Delta\subset \Symm_m(\CC)$ is
the set of singular matrices, $\Delta=V(\det)$.
\item
\label{cond:stab}
The generic stabilizer of this action is trivial when $n=m$ and otherwise
isomorphic to $\Orthog(n-m,\CC)$. 
\item
\label{cond:lift}
All vector fields $\eta\in\Der_S(-\log \Delta)$ lift.
\end{enumerate}
\end{lemma}
\begin{proof}
Differentiating $t\mapsto \varphi(A+tB)$
at the origin 
shows that
$d\varphi_{(A)}(B)=A^TB+B^TA$.
If $\rank(A)=m$ and  
$C\in \Symm_m(\CC)$
that we identify with the tangent space $T_{\varphi(A)} \Symm_m(\CC)$,
then there exists a $K\in M(n,m,\CC)$ such that
$K^TA=\frac{1}{2}C$, and so
$d\varphi_{(A)}(K)=C$. 
If $\rank(A)<m$, then let $v$ be a nonzero column vector in $\ker(A)$.
For any $B$, $v^T(A^TB+B^TA)v=0$, but 
it is easy to produce some 
$C\in\Symm_m(\CC)$ with $v^TCv\neq 0$.
This proves \eqref{cond:submersion}.

The first part of \eqref{cond:crit} follows from 
\eqref{cond:submersion} and linear algebra.
Since $\rank(\varphi(A))\leq \rank(A)$, 
we have
$\Delta\subseteq V(\det)$.
If $B\in V(\det)$, then 
let $D=G^TBG$ be the diagonalization of $B$ as the matrix of a
symmetric bilinear form.
If $H=G^{-1}$, then
since $D$ has diagonal entries in $\{0,1\}$, $D=D^2$ and
$B=H^TDH=(DH)^T(DH)$.
Appending zeros to the bottom of the $m\times m$ $DH$ produces an
$A\in X$ with $\varphi(A)=B$ and $\rank(A)=\rank(B)$.
Thus $V(\det)\subseteq \Delta$.

\eqref{cond:stab} follows from computing the stabilizer at
$\begin{pmatrix} I \\ 0 \end{pmatrix}\in X$.

For \eqref{cond:lift}, observe that
$\GL(m,\CC)$ has representations $\rho_X$ and $\rho_S$ on $X$ and $S$
defined by
$$\rho_X(A)(B)=BA^T\qquad \textrm{and}\qquad\rho_S(A)(C)=ACA^T.$$
Since $\varphi$ is equivariant with respect
to these representations, the vector fields from $\rho_S$ lift by
Lemma \ref{lemma:grouplift}.
By a free resolution due to 
J\'ozefiak
(see \cite[\S3.2]{goryunovmond}),
the vector fields from $\rho_S$ generate the module
$\Der_S(-\log \det)$.
It follows that all elements of $\Der_S(-\log \det)$ are liftable. 
\end{proof}

For the purposes of applying
Proposition \ref{prop:coregulart1cm} and Theorem \ref{mthm},
the case where $n=m+1$ is particularly nice.

\begin{proposition}
\label{prop:on}
Let $G=\Orthog(m+1,\CC)$ act on $X=M_{m+1,m}$ by
multiplication on the left, and let
$\varphi:X\to S=X\slashslash G\cong\Symm_m(\CC)$ be the quotient map.
If $f$ defines a free divisor in $S$ which contains the hypersurface
of singular matrices in $\Symm_m(\CC)$,
then
$f\circ \varphi$ defines a free divisor in $X$.
\end{proposition}
\begin{proof}
We check the hypotheses of Theorem \ref{mthm}.

By Lemma \ref{lemma:onproperties}\eqref{cond:crit}
and \eqref{cond:stab},
the critical locus 
$C(\varphi)$ has codimension $2$ and the
generic stabilizer has dimension $0$.
Hence by Proposition \ref{prop:coregulart1cm}, $T^1_{X/S}$ is Cohen--Macaulay
of codimension $2$, giving us \eqref{mthm.t1xs}.

Since $f=0$ contains the singular matrices, 
$\Der_S(-\log f)\subseteq \Der_S(-\log \det)$,
and all of these vector fields lift by Lemma
\ref{lemma:onproperties}\eqref{cond:lift}.
Thus we have \eqref{mthm.ks}.
\end{proof}

Theorem \ref{altmthm} may also produce free divisors from the
square case ($n=m$), provided 
we can prove $T^0_{X/S}$ is free.

\begin{example}
\label{ex:square}
Let $G=O(m,\CC)$ act on $X=M_{m,m}$ by multiplication on the left,
with coregular quotient
$\vp:X\to S\cong \Symm_m(\CC)$.
By Lemma \ref{lemma:onproperties}\eqref{cond:crit},
$\dim(T^1_{X/S})=\dim(X)-1$.
Although the proof of Proposition \ref{prop:coregulart1cm}
shows $\rho:\fg\otimes\calox\to T^0_{X/S}$ is an
isomorphism off $C(\vp)=\supp(T^1_{X/S})$,
it no longer automatically follows that $\rho$ is an isomorphism
everywhere. 
For $m\leq 8$, 
Macaulay2 calculations show that 
$T^0_{X/S}$ is free and $\rho$ is an isomorphism.
(As predicted by Proposition \ref{prop:numericalcriterion},
$\dim(X)=m^2\neq \sum_{\nu=1}^{\binom{m+1}{2}} 2$.)
Lemma \ref{lemma:onproperties}\eqref{cond:lift}
identifies the liftable vector fields.
Thus, for a free divisor in $S$
containing the singular matrices in $\Symm_m(\CC)$,
conditions
\eqref{mthm.t0xsiscm}
and \eqref{mthm.ks}
of Theorem \ref{altmthm} are satisfied,
and \eqref{mthm.singloc} is easy to check.

For instance,
use coordinates
$\begin{pmatrix} s_{11} & s_{12} \\ s_{12} & s_{22} \end{pmatrix}$
for $\Symm_2(\CC)$
and
$\begin{pmatrix} x_{11} & x_{12} \\ x_{21} & x_{22} \end{pmatrix}$
for $M_{2,2}$.
Then the free divisor defined by $f=s_{11}(s_{11}s_{22}-s_{12}^2)$
satisfies
\eqref{mthm.singloc} of Theorem \ref{altmthm}, and so the reduction
of 
$f\vp=(x_{11}^2+x_{21}^2)(x_{11}x_{22}-x_{12}x_{21})^2$
defines a free divisor in $X$.
Indeed, a change of coordinates on $X$ takes
$\frac{f\vp}{\det}$
to the well-known example
$x_{11}x_{12}(x_{11}x_{22}-x_{12}x_{21})$.
Similar examples for higher $m$ have been exhibited 
by David Mond.

\end{example}

\subsection*{Symplectic group}

\begin{sit}
Let $n$ be even and let $G=\Sp(n,\CC)\subseteq \GL(n,\CC)$ be the
symplectic group acting on the space $V=M_{n,m}$ by multiplication on the
left.
Let $\Omega=\begin{pmatrix} 0 & I \\ -I & 0 \end{pmatrix}$,
for $I$ the identity matrix, and
let $(u,v)=u^T\Omega v$ denote the skew-symmetric bilinear form preserved by
$G$.
By \cite[\S9.3]{PV}, the ring of invariants is generated by the
$\binom{m}{2}$ degree $2$ polynomials of the form $(u,v)$, where
$u$ and $v$ are columns.
If $\Sk_m(\CC)$ denotes the space of $m\times m$ complex skew-symmetric  
matrices, these are equivalently the polynomials given by the entries of
$\vp:X=M_{n,m}\to S=\Sk_m(\CC)$ defined by
$\vp(A)=A^T\Omega A$.
By \cite[\S9.4]{PV}, there are relations between these generators precisely
when $m\geq n+2$.

Thus assume that $1<m\leq n+1$,
so that $V\slashslash G\cong \Sk_m(\CC)$ is coregular and the quotient
is given by $\vp$.
Recall that the rank of
any $C\in\Sk_m(\CC)$ is always even, and that the
square of the
Pfaffian
$\Pf:\Sk_m(\CC)\to\CC$
is equal to the determinant function.
We prove some basic properties.
\end{sit}

\begin{lemma}
\label{lemma:spproperties}
Let $\vp$ be as above, with $1<m\leq n+1$.
\begin{enumerate}[\rm(i)]
\item
\label{cond:spsubmersion}
For $A\in X$, $\vp$ is a submersion at $A$ if and only if $\rank(A)\geq m-1$.
\item
\label{cond:spcrit}
$C(\vp)=\{A\in X:\rank(A)<m-1\}$ has codimension $2(n-m+2)$.
The discriminant of $\vp$ is $\Delta=\{C\in\Sk_m(\CC): \rank(C)<m-1\}$.
\item
\label{cond:spvfs}
$\Der_S(-\log \Delta)$ is generated by the linear vector fields
coming from the $\GL(m,\CC)$ action $A\cdot C=ACA^T$.
\item
\label{cond:splift}
All vector fields from $\Der_S(-\log \Delta)$ lift across $\vp$.
\item
\label{cond:spstab}
The dimension of the generic stabilizer is
$\frac{1}{2}\left((n-m)^2+m\right)$ when $1\leq m\leq \frac{n}{2}$,
and
$\frac{1}{2}\left((n-m)^2+n-m\right)$ when $\frac{n}{2}\leq m\leq n+1$.
\end{enumerate}
\end{lemma}

First we prove a lemma.
\begin{lemma}
\label{lemma:rankconstruction}
Let $B\in M_{m,n}$ have rank $\geq m-1$.
Then for any $C\in \Sk_m(\CC)$ there exists a $D\in M_{n,m}$ and
$E\in\Symm_m(\CC)$ such that
$C=BD+E$.
\end{lemma}
\begin{proof}
The $\rank(B)=m$ case is clear, with $E=0$.
Let $\rank(B)=m-1$, and
let $v\notin \mathrm{Image}(B)$. 
By our rank assumption, every $z\in M_{m,1}$ is the sum of
an element of $\mathrm{Image}(B)$ and a multiple of $v$.
Applying this to each column of $C$,
there exists an $A\in M_{n,m}$
and a $w\in M_{m,1}$ such that
$C=BA+vw^T$.
Now write $w=Bu+\lambda v$,
for $u\in M_{n,1}$ and $\lambda\in \CC$.
Then as required,
\begin{equation*}
C=B(A-uv^T)+\left((Buv^T)+(Buv^T)^T+\lambda vv^T\right).\qedhere
\end{equation*}
\end{proof}

\begin{proof}[Proof of \ref{lemma:spproperties}]
Differentiating $t\mapsto A+tD$ at $0$ shows that
$d\vp_{(A)}(D)=A^T\Omega D+D^T\Omega A$.

If $\rank(A)\geq m-1$, then $\rank(A^T\Omega)\geq m-1$.  Let
$C\in\Sk_m(\CC)$, and by Lemma \ref{lemma:rankconstruction}
write
$\frac{1}{2}C=A^T\Omega D+E$, where $E\in\Symm_m(\CC)$.  Then
$d\vp_{(A)}(D)=C$.
If $\rank(A)<m-1$, then let $v,w$ be two linearly independent vectors in $\ker(A)$.
Although $v^T(d\vp_{(A)}(D))w=0$ for any $D$, there exists 
$C\in\Sk_m(\CC)$ such that $v^TCw\neq 0$:
if $v=\sum v_ie_i$ and $w=\sum w_ie_i$ are expressed in terms of a basis and
$E_{ij}$ is an elementary matrix (with one nonzero entry), then
$v^T(E_{ij}-E_{ij}^T)w=v_iw_j-v_jw_i$.
Since $v,w$ are linearly independent,
this proves \eqref{cond:spsubmersion}.

The first part of \eqref{cond:spcrit} follows from
\eqref{cond:spsubmersion} and linear algebra.
If $A\in C(\vp)$, then $\rank(A^T\Omega A)< m-1$, and hence
the discriminant $\Delta$
is contained in the claimed set.

The converse will follow by showing that if $C\in \Sk_m(\CC)$ has 
rank $2r$, then there exists an $A\in M_{n,m}$ of rank $2r$ with
$\vp(A)=C$.
By the standard form of
skew-symmetric bilinear forms, there exists a $K\in\GL(m,\CC)$
such that $K^TCK$ is block diagonal, with $r$ blocks of the form
$J=\begin{pmatrix} 0 & 1\\ -1 & 0 \end{pmatrix}$ and the remainder 
zero.
There exists a permutation matrix $P$ such that $P^T\Omega P$ is block
diagonal, with $\frac{n}{2}$ copies of $J$.
Let $B\in M_{n,m}$ be zero, except with a copy of the identity in
the upper left $2r\times 2r$ submatrix
(it fits when $m=n+1$ because $m$ is then odd and hence $2r\leq n$, and also fits when $m\leq
n$).
A calculation shows that
$B^T P^T\Omega P B=K^T C K$,
and hence $\vp(PBK^{-1})=C$ with $\rank(PBK^{-1})=2r$.

For \eqref{cond:spvfs}, since this action preserves all rank varieties in $S$,
such vector fields are in $\Der_S(-\log \Delta)$.
When $m$ is even, then $\Delta$ is defined by the Pfaffian $\Pf$.
By a free resolution due to J\'ozefiak--Pragacz 
(see \cite[\S3.3]{goryunovmond}),
the vector fields from the action generate
$\Der_S(-\log \Pf)$. 

When $m$ is odd, then $\Delta$ is defined by the ideal
$I=(P_1,\ldots,P_m)$,
where $P_i$ is the Pfaffian after deleting row $i$
and column $i$.
Let $\eta'\in \Der_S(-\log \Delta)$,
so that $\eta'(P_i)\in I$ for all $i$.
Calculations show that for all $i,j$, there exists a linear vector field
$\xi_{ij}$ coming from the action such that $\xi_{ij}(P_k)$
is $0$ if $i\neq k$ and $P_j$ if $i=k$.
%
%
%
%
An appropriate linear combination added to
$\eta'$ will produce an $\eta$ which annihilates
each of $P_1,\ldots,P_m$.
Let $\pi:\Sk_{m+1}(\CC)\to \Sk_m(\CC)$ be the projection which deletes
the last row and column.
If $\eta=\sum_{1\leq i<j\leq m} \alpha_{ij} \frac{\partial}{\partial
x_{ij}}$, then
let $\tilde{\eta}=\sum_{1\leq i<j\leq m} \alpha_{ij}\circ
\pi\frac{\partial}{\partial x_{ij}}$ be a vector field on
$\Sk_{m+1}(\CC)$.
A calculation shows that $\tilde{\eta}$ must annihilate
$\Pf$ on $\Sk_{m+1}(\CC)$.
By the even case, $\tilde{\eta}$ may be written in terms of
the linear vector fields coming from the
$\GL({m+1},\CC)$ action on $\Sk_{m+1}(\CC)$.
Since $\tilde{\eta}$ does not depend on
the last column,
the coefficients of the linear vector fields
may be restricted to functions on $\Sk_{m}(\CC)$
and the corresponding linear vector fields on $\Sk_m(\CC)$ used,
to express $\eta$ in terms of the linear vector fields coming from the
$\GL(m,\CC)$ action.
This proves \eqref{cond:spvfs}, and \eqref{cond:splift} follows just
as for Lemma \ref{lemma:onproperties}\eqref{cond:lift}.

For \eqref{cond:spstab}, the Lie algebra of the isotropy subgroup at
$P=\begin{pmatrix} I_m \\ 0 \end{pmatrix}$ when $m\leq n$,
or at
$P=\begin{pmatrix} I_n & 0 \end{pmatrix}$ when $m=n+1$,
is straightforward to compute by considering the 
cases $1\leq m\leq \frac{n}{2}$, $\frac{n}{2}\leq m\leq n$, and
$m=n+1$.
\end{proof}
When $m=n+1$, by Lemma \ref{lemma:spproperties},
Proposition \ref{prop:coregulart1cm}, and Theorem
\ref{mthm}, we have 

\begin{proposition}
Let $n$ be even.
Let $G=\Sp(n,\CC)\subseteq \GL(n,\CC)$ act on $X=M_{n,n+1}$ by multiplication on the left,
and let
$\vp:X\to S\slashslash G\cong \Sk_{n+1}(\CC)$ be the quotient map.
Let $\Delta$ be as in Lemma \ref{lemma:spproperties}.
If $f$ defines a free divisor in $S$ 
for which $\Der_S(-\log f)\subseteq \Der_S(-\log \Delta)$,
then $f\circ \vp$ defines a free divisor in $M_{n,n+1}$.
\qed
\end{proposition}

\bibliographystyle{amsalpha}
\bibliography{refs}

\end{document}